\documentclass[11pt]{article}

\usepackage[utf8]{inputenc}
\usepackage[T1]{fontenc}

\usepackage{amsmath}
\usepackage{graphicx,psfrag,epsf}
\usepackage{enumerate}
\usepackage[square,numbers,sort&compress]{natbib}
\usepackage{url}

\usepackage{snapshot}

\usepackage[left=1.25in,right=1.25in, top=1.2in, bottom=1.2in]{geometry} 

\usepackage{multirow}

\usepackage{graphicx}
\usepackage{amssymb, amsmath, amsthm}

\usepackage{txfonts}

\usepackage[ruled]{algorithm2e}

\usepackage[font=small,labelfont=bf]{caption}

\usepackage{comment}

\usepackage[font=footnotesize,labelfont=bf]{subcaption}
\usepackage{booktabs}
\usepackage{csquotes}
\usepackage{url}

\newtheorem{theorem}{Theorem}
\newtheorem{lemma}[theorem]{Lemma}

\newtheorem{cor}[theorem]{Corollary}

\newtheorem{example}[theorem]{Example}

\DeclareMathOperator{\traj}{\operatorname{traj}}

\usepackage{epstopdf}
\newcommand{\deltaSNR}{\mathrm{\Delta SNR}}
\newcommand{\dB}{\mathrm{\,dB}}
\newcommand{\Pos}{\mathrm{Pos}}

\title{
Variational Regularization of Inverse Problems for Manifold-Valued Data 
}
\author{Martin Storath%
\thanks{Image Analysis and Learning Group, Interdisciplinary Center for Scientific Computing, Universit\"at Heidelberg, Germany.},
Andreas Weinmann\thanks{Department of Mathematics and Natural Sciences, Hochschule Darmstadt, and Institute of Computational Biology, Helmholtz Zentrum München, Germany. 
}} 
\date{\today}

\usepackage{customcommands}

\newcommand{\dist}{\mathrm{dist}}
\newcommand {\mean}{\mathrm{mean}}
\newcommand {\TV}{\mathrm{TV}}
\newcommand {\TGV}{\mathrm{TGV}}
\newcommand {\diam}{\mathrm{diam}}
\newcommand {\midp}{\mathrm{mid}}
\newcommand {\pt}{\mathrm{pt}}
\newcommand {\prox}{\mathrm{prox}}

\newlength\figureheight
\newlength\figurewidth
\begin{document}
\maketitle

\begin{abstract}
	In this paper, we consider the variational regularization of manifold-valued data in the inverse problems setting.
	In particular, we consider TV and TGV regularization for manifold-valued data with indirect measurement operators. 
	We provide results on the well-posedness and present algorithms for a numerical realization of these models in the manifold setup. Further, we provide experimental results for synthetic and real data to show the potential of the proposed schemes for applications. 
\end{abstract}

\section{Introduction}

In various problems of applied sciences the data do not take values in a linear space but in a nonlinear space such as a smooth manifold. Examples of manifold-valued data are circle and sphere-valued data as appearing in interferometric SAR imaging \cite{massonnet1998radar}
and in color image processing \cite{chan2001total,vese2002numerical,kimmel2002orientation, lai2014splitting},
or as wind directions \cite{schultz1990circular} and orientations of flow fields \cite{adato2011polar, storath2017fastMedian}.
Other examples are data taking values in the special orthogonal group $\mathrm{SO}(3)$ expressing vehicle headings, aircraft orientations or camera positions \cite{rahman2005multiscale},
Euclidean motion group-valued data \cite{rosman2012group} as well as shape-space data \cite{Michor07,berkels2013discrete}.
Another prominent manifold is the space of positive (definite) matrices
endowed with the Fisher-Rao metric  \cite{radhakrishna1945information}. 
It is a Cartan-Hadamard manifold which has nice differential-geometric properties.
Moreover, it is the data space in diffusion tensor imaging (DTI)  \cite{pennec2006riemannian},
a medical imaging modality which allows the quantification of diffusional characteristics of a specimen non-invasively \cite{basser1994mr,johansen2009diffusion}, and thus is helpful in connection with
neurodegenerative pathologies such as schizophrenia  and autism 
\cite{foong2000neuropathological, alexander2007diffusion}. 
Related work including various means of processing manifold-valued data is discussed in Section~\ref{sec:RelatedWirk} below.

In this paper, we consider regularizing the manifold valued analogue of the discrete inverse problem
$ Au = f,$
where $A$ is a matrix with unit row sums (but potentially negative items of $A$), $f$ is given data and $u$ is the objective to reconstruct. 
More precisely, we consider the manifold analogue of Tichanov-Philips regularization 
which, for discrete data $f=(f_i)_{i=1}^N$,
reads
\begin{align}\label{eq:TichPhEuclidean}
\argmin_{u \in \mathbb{R}^K}  \left\|Au-f \right\|^p_p    +  \lambda \  R(u)
\ = \ 
  \argmin_{u \in \mathbb{R}^K} \  \sum\nolimits_{i=1}^N \left|\sum\nolimits_j A_{i,j}u_j-f_i \    \right|^p    +  \lambda \  R(u)
\end{align}
where $R$ is a regularizing term typically incorporating some prior assumption on the data.
To this end, we replace the euclidean distance $|\cdot -f|$ by the Riemannian distance $\dist(\cdot,f)$ and 
the weighted mean $\sum_j A_{i,j}u_j$ by the weighted Riemannian center of mass \cite{karcher1977riemannian, kendall1990probability}
denoted by $\mean(A_{i,\cdot},u)$
which is given by 
\begin{align}\label{eq:IntrMeanIntro}
	\mean(A_{i,\cdot},u) = \argmin_{v \in \mathcal{M}} \ \sum\nolimits_j  A_{i,j} \ \dist(v,u_j)^2.
\end{align}
We propose to consider the following natural manifold analogue of the variational problem \eqref{eq:TichPhEuclidean} which reads
\begin{align}
\argmin_{u \in \mathcal{M}^K} \sum\nolimits_{i=1}^N\dist(\mean(A_{i,\cdot},u),f_i)^p + \lambda \  R(u).
\end{align}
Here $R(u)$ is a regularizing term, for instance 
\begin{align}
 R(u) = \TV(u), \qquad \text{or} \qquad  R(u) = \TGV(u),
\end{align}
where $\TV(u)$ denotes the total variation 
and $\TGV(u)$ denotes the total generalized variation of 
discrete manifold valued target $u.$
We note that our setup in particular includes the manifold analogue of positive convolution operators, e.g., modeling blur, in one, two or several dimensions of domain.

\subsection{Related Work}\label{sec:RelatedWirk}

In recent years, there has been a lot of work on processing manifold valued data in the applied sciences.
For instance, wavelet-type multiscale transforms for manifold-valued data are considered in
\cite{rahman2005multiscale,grohs2009interpolatory, weinmannConstrApprox, wallner2011convergence, weinmann2012interpolatory}.
Manifold-valued partial differential equations have for instance been considered in   
\cite{tschumperle2001diffusion, chefd2004regularizing}.
In particular, finite element methods for manifold-valued data are the topic of \cite{grohs2013optimal,sander2015geodesic}.
Work on statistics on Riemannian manifolds can be found in 
\cite{oller1995intrinsic,bhattacharya2003large,fletcher2004principal,
	bhattacharya2005large,pennec2006intrinsic,fletcher2007riemannian}.  
Optimization problems for manifold-valued data are for example the topic of \cite{absil2009optimization,absil2004riemannian},
of \cite{grohs2016varepsilon} 
and of \cite{hawe2013separable} with a view towards learning in manifolds. We also mention related work on optimization in Hadamard spaces \cite{bacak2014convex, bavcak2013proximal} and on the Potts and Mumford-Shah models for manifold-valued data  \cite{weinmann2015mumford,storath2017jump}.  

The first definition of the Riemannian center of mass can be traced back to Fr\'echet; 
a lot of work has been done by Karcher \cite{karcher1977riemannian}. For details including an historic overview we refer to \cite{afsari2011riemannian}; see also \cite{kendall1990probability}. Due to its use as a means of averaging in a manifold,
it is employed as a basic building block in various works; e.g., \cite{pennec2006riemannian,fletcher2007riemannian,tron2013riemannian}
as well as many references in the paragraph above. 
Various contributions deal with the numerical computation of the Riemannian center of mass; 
see, for instance \cite{afsari2013convergence,ferreira2013newton,arnaudon2013medians}.

Concerning the variational denoising of signals and images, or data, there has been a lot of work on extending TV regularization \cite{rudin1992nonlinear} and related higher order methods such as TGV regularization \cite{Bredies10} to the setting of manifold-valued data in recent years.
TV functionals for manifold-valued data have been considered from the analytical side in   
~\cite{GM06,GM07,GMS93}; 
in particular, the existence of minimizers of certain TV-type energies has been shown. 
A convex relaxation based algorithm for TV regularization for $\mathbb S^1$-valued data was considered in~\cite{SC11,CS13}. 
Approaches for TV regularization for manifold-valued data are considered in \cite{LSKC13}
which proposes a reformulation as multi-label optimization problem and a convex relaxation,
in \cite{grohs2016total} which proposes iteratively reweighted minimization, 
and in \cite{weinmann2014total} which proposes cyclic and parallel proximal point algorithms.
An exact solver for the TV problem for circle-valued signals has been proposed in \cite{storath2016exact}.
Applications of TV regularization to shape spaces may be found in \cite{baust2015total}.
As with vector space data, TV regularization for manifold-valued data has a tendency to produce piecewise constant results
in regions where the data is smooth. 
As an alternative which prevents this, second-order TV type functionals for circle-valued data have been considered in \cite{bergmann2014second,bergmann2016second} and, for general manifolds, in \cite{bavcak2016second}. 
However, similar to the vector space situation, 
regularization with second-order TV type functionals tends towards producing solutions which do not preserve the edges as desired.
To address this drawback, TGV regularization for manifold-valued data has been introduced in \cite{bredies2017total}.
In the aforementioned works, only direct data terms are considered. Concerning indirect data, there is the author's works
\cite{stefanoiu2016joint} and \cite{baust2016combined} where very specific (pixel-based) data terms are considered.
There a forward-backward type algorithm is proposed to solve the corresponding inverse problem.
In this paper, we particularly focus on incorporating the manifold analogue of the operator $A$ into the manifold setting of $\TV$ and $\TGV$ regularization (as well as classical $H^1$ regularization and related terms.) To our knowledge, this has not been done before.

\subsection{Contributions}

The contributions of the paper are as follows: \emph{(i)}
 we study models for variational (Tichanov-Phillips) regularization
 for indirect measurement terms  in the manifold setup; 
 \emph{(ii)} we provide algorithms for the proposed models;
\emph{(iii)} we show the potential of the proposed algorithms by applying them to synthetic and real data.
Concerning \emph{(i)}, we derive models for variational (Tichanov-Phillips) regularization
of indirectly measured data in the manifold setup 
which in particular apply to deconvolution/deblurring and to $\TV$ and $\TGV$ regularization for manifold-valued data in a multivariate setting. 
We obtain well-posedness results for the variational problems, i.e., 
results on the existence of minimizers.
Concerning \emph{(ii)}  we provide the details for algorithmic realizations of the proposed variational
models. 
For differentiable data terms, we build on the concept of a generalized forward backward-scheme.
We extend it by a trajectory method and a Gau\ss-Seidel type update scheme
which significantly improves the performance compared with the basic scheme.
We also consider a variant based on a stochastic gradient descend part.   
For a non-differentiable data term, we employ the well-established concept of a cyclic proximal point algorithm. 
A challenging part and a central contribution consists in the computation of the (sub)gradients of the 
data terms. This involves deriving rather explicit representations of the derivative of the intrinsic mean mapping \eqref{eq:IntrMeanIntro}
with respect to the input points in the manifold
where for the intrinsic mean mapping \eqref{eq:IntrMeanIntro} there is no closed-form solution available.
Concerning \emph{(iii)}, we provide a numerical study of the proposed schemes.
We provide experiments with real and synthetic data living on the unit circle, in the two-dimensional sphere as well as in the space of positive matrices.

\subsection{Outline of the Paper}

The paper is organized as follows. 
The topic of Section~\ref{sec:model} is to derive a model for variational (Tichanov-Phillips) regularization
of indirectly measured data in the manifold setup. 
Section~\ref{sec:well-posedness} deals with the well-definedness of the proposed model.
The topic of Section~\ref{sec:algorithm} is the algorithmic realization of the proposed models.
In Section~\ref{sec:Experiments} we provide a numerical investigation of the proposed algorithms.
Finally, we draw conclusions in Section~\ref{sec:Conclusion}.

\section{Model}\label{sec:model}

 In Section~\ref{sec:TichPhilManifold} we give a detailed formulation of the manifold analogue of the variational regularization problem \eqref{eq:TichPhEuclidean} briefly explained in the introduction. In  Section~\ref{sec:Appl2Deblurring}, we consider deconvolution problems. The topic of Section~\ref{sec:SpecificRegularizers} are specific regularizers; in particular, $\TV$ and $\TGV$ regularizers.

\subsection{Variational Regularization of Inverse Problems in the Manifold Setup -- Problem Formulation}\label{sec:TichPhilManifold}

In this paper, we always consider a complete and connected Riemannian manifold $\mathcal M.$ We further consider a matrix $A$ with entries $a_{ij},$ i.e.,
\begin{equation}
	A = 
	\begin{pmatrix}
		a_{11}  & \cdots & a_{1K}  \\
		\vdots  &        & \vdots  \\
		a_{N1}  & \cdots & a_{NK},         
	\end{pmatrix}	
\end{equation} 
with unit row sums, i.e., $\sum_j a_{ij} = 1$ for all $i=1,\ldots,N.$ 
(Note that we do not require the particular items $a_{ij}$ to be nonnegative.) 
Then, the operator $\mathcal A$ induced by the kernel $A,$ acting on signals $u \in \mathcal M^k,$ is given by

\begin{align}\label{eq:DefMatVecAnaManiIth}
 \mathcal A(u)_i = \mean(a_{i,\cdot},u) = \argmin_{v \in \mathcal{M}} \ \sum\nolimits_j  a_{i,j} \dist(v,u_j)^2,
\end{align}
where $\dist$ is the Riemannian distance function induced by the Riemannian metric in $\mathcal M.$
Note that $\mathcal A$ yields an element $\mathcal A(u) \in \mathcal M^N.$ The described operation is just a suitable analogue of 
matrix-vector multiplication for data living in a manifold.
The mean  $\mean(a_{i,\cdot},u)$ is not unique for all input configurations. 
However, for data living in a small enough ball uniqueness is guaranteed. 
For a discussion, precise bounds and an extensive list of references concerning this interesting line of research we refer to~\cite{afsari2011riemannian}. 
A reasonable way to deal with the nonuniqueness is to consider the whole set of minimizers in such a case.
In this paper we do so where -- in case of nonuniqueness -- we add an additional constraint originating from our variational formulation 
as explained at the end of this subsection (after deriving the variational formulation).

In order to reconstruct a signal $u$ from given data $f \approx \mathcal A (u),$ we consider the manifold-valued analogue of Tichanov-Phillips regularization, i.e., the variational problem 
\begin{align}\label{eq:TichManiGeneral}
	\argmin_{u \in \mathcal{M}^K} \dist(\mathcal A(u),f)^p + \lambda \  R(u).
\end{align}
Here, 
\begin{align}\label{eq:DefDist4Vectors}
    \dist(\mathcal A(u),f)^p =  \sum\nolimits_{i=1}^N\dist(\mathcal A(u)_i,f_i)^p, \qquad p \in [1,\infty),
\end{align}
and $R$ is a regularizer. Examples of regularizing terms are the manifold analogues of classical $H^1$ terms, of $\TV$ and higher order generalizations such as $\TGV.$
Details on specific regularizers are given in Section~\ref{sec:SpecificRegularizers} below. In this paper, we are particularly interested in incorporating the operator $\mathcal A$ into the manifold setting of $\TV$ and $\TGV$ regularization 
(as well as classical $H^1$ regularization and related terms.) 

As pointed out above the set of minimizers in \eqref{eq:DefMatVecAnaManiIth} may be non-unique. In such a case and in view of \eqref{eq:DefDist4Vectors} we overload the definition \eqref{eq:DefMatVecAnaManiIth} in the context of \eqref{eq:TichManiGeneral} by
\begin{align}\label{eq:DefMatVecAnaManiOverload}
\mathcal A(u)_i = \{v \in \mathcal M: v \text{ minimizes \eqref{eq:DefMatVecAnaManiIth} and } v \text{ minimizes } \dist(v,f_i)
\text{ in } \mean(a_{i,\cdot},u).\}
\end{align}
which means that we choose the average closest to the data $f_i.$ (This seems reasonable since the chosen averages intuitively explain the data best.)

\subsection{Manifold Analogues of Deconvolution Problems}\label{sec:Appl2Deblurring}

Shift invariant linear systems are particularly interesting when processing data in a linear space; 
for instance, blur is often modeled 
by a shift invariant system; further, technical systems are often built in a way to (approximately) fulfill shift invariance. 
In the setup of real-valued data, a shift invariant system is essentially (up to boundary treatment) described by a matrix $A$
which is diagonally constant, i.e. $a_{i,j} = a_{i-1,j-1}=a'_{i-j},$ for all $i,j$ 
(where, with boundary treatment, one gets cyclic matrices or cut-off Toeplitz matrices). 
Here, $a'$ denotes the corresponding convolution kernel. Using \eqref{eq:DefMatVecAnaManiIth} for the resulting matrix $A$ (with the weights  $a_{i,j} = a'_{i-j}$ given by the corresponding convolution kernel $a'$) in \eqref{eq:TichManiGeneral}, we obtain the manifold analogue 
of the deconvolution problem for $a'.$ 
 
We explicitly point out that the definition in \eqref{eq:DefDist4Vectors} which implements the first term in 
\eqref{eq:TichManiGeneral} is not restricted to a univariate setting but does refer to a multivariate general setting:
either  vectorize the multivariate data and obtain a corresponding resulting matrix $A$ defining the operator $\mathcal A,$ 
or just read the index $i$ in \eqref{eq:DefDist4Vectors} as a multiindex. For instance, 
using the multiindex notation $i=(k,l),$ $j=(r,s)$, 
the manifold analogue $\mathcal A$ of the  convolution operator $A$ given by 
\begin{align}
Au_{j} = Au_{rs} =  \sum_{k,l}  a'_{r-k, s-l} \ u_{kl} = \sum_i a'_{j-i}  u_{i}
\end{align}
reads
\begin{align}\label{eq:ConvMani2D}
\mathcal A(u)_{j} =  \mathcal A(u)_{r,s} = \argmin_{v \in \mathcal{M}} \ \sum_{k,l}  a'_{r-k, s-l}  \ \dist(v,u_{kl})^2,
\end{align}
where $a'$ denotes the convolution kernel. 
Hence, eq.\ \eqref{eq:TichManiGeneral} with data term \eqref{eq:ConvMani2D} constitutes the explicit variational problem formulation for the multivariate deconvolution problem for manifold-valued data.

\subsection{Problems with Specific Regularizers} \label{sec:SpecificRegularizers}

A very prominent method for regularizing images is the Rudin-Osher-Fatemi (ROF) model \cite{rudin1992nonlinear} which employs TV regularization. 
Using a TV term as regularizer in \eqref{eq:TichManiGeneral} we obtain a manifold version with indirect data term of the discretized ROF model in its Lagrange formulation.  It reads    
\begin{align}\label{eq:DefTivchManiTV}
\argmin_{u \in \mathcal{M}^K} \dist(\mathcal A(u),f)^p + \lambda \  \TV(u).
\end{align}
In the univariate setting, 
$
\TV(u) = \sum_l \dist(u_l,u_{l-1}).
$
In the bivariate setting we employ the anisotropic discretization 
\begin{align}\label{eq:DefTVreg}
R(u) =  \TV(u) = \sum\nolimits_{l,k} \dist(u_{l,k},u_{l-1,k})  +  
\sum\nolimits_{l,k} \dist(u_{l,k},u_{l,k-1}).
\end{align}
In the manifold context, these regularizers have been considered in \cite{weinmann2014total}.
Anisotropy effects resulting from this discretization can be reduced by incorporating further directional difference (such as diagonal differences); for a discussion, we refer to \cite{weinmann2015mumford}.
The generalizations for the general multivariate setup are obvious.

A more general first order regularizing term 
-- also including the (discrete) manifold analogues of classical $H^1$ Hilbert-Sobolev seminorm-regularization --
is easily obtained by letting
\begin{align}\label{eq:VqRegulizer}
R(u) = V^q(u) = \sum\nolimits_{l,k} \dist(u_{l,k},u_{l-1,k})^q  +  
\sum\nolimits_{l,k} \dist(u_{l,k},u_{l,k-1})^q, \qquad q \in [1,\infty).
\end{align}
We then consider the model \eqref{eq:TichManiGeneral} with regularizer $R$ given by 
\eqref{eq:VqRegulizer}.
Classical $H^1$ regularization corresponds to letting $q=2,$ whereas $\TV$ regularization corresponds to $q=1.$

A prominent topic in the variational regularization of image and imaging data is to reduce the staircasing effects observed when using the $\TV$ regularizer. A first means to do so, is to add a second order $\TV$ term, i.e., to consider the 
regularizer $R(u)$ defined by 
\begin{align}\label{eq:mixedFirstSecond}
R(u) = \mu_1 \TV(u) + \mu_2 \TV_2(u), \qquad  \mu_1, \mu_2 \geq 0.  
\end{align}
For manifold-valued data we have considered this regularizer in \cite{bergmann2016second} as well as in 
\cite{bergmann2014second} for data living on the unit circle. In the univariate setting, the second order TV type term is given by 
$
 \TV_2(u) = \sum_i d_2(u_{i-1},u_{i},u_{i+1}),
$
where 
$
d_2(u_{i-1},u_{i},u_{i+1}) = \dist (u_{i},\midp(u_{i-1},u_{i+1})) 
$
and $\midp(u_{i-1},u_{i+1})$ denotes the (geodesic) midpoint of $u_{i-1}$ and $u_{i+1}.$ In the multivariate setting for manifold valued data,
\begin{align}\notag
\mu_2 \TV_2(u) = 
\mu_{2,1} \sum\nolimits_{l,k}  d_2(u_{l-1,k},u_{l,k},u_{l+1,k})  
&+ \mu_{2,2} \sum\nolimits_{l,k} d_{1,1}(u_{l-1,k-1},u_{l,k-1},u_{l-1,k},u_{l,k})\\ 
&+ \mu_{2,1} \sum\nolimits_{l,k} d_2(u_{l,k-1},u_{l,k},u_{l,k-1})\label{eq:DefTV2}
\end{align}
with the weights $\mu_{2,1},\mu_{2,2}$ and the diagonal second order terms $d_{1,1}$ which are given by
$d_{1,1}(u_{l-1,k-1},$ $u_{l,k-1},u_{l-1,k},u_{l,k})$ 
$= \dist (\midp(u_{l,k-1},u_{l-1,k}),\midp(u_{l-1,k-1},u_{l,k})).$

In order to improve the quality of the results further, 
we have considered TGV regularization (with a direct data term) in the manifold setting in \cite{bredies2017total}. 
The TGV functional with an indirect measurement term in a discrete setting for manifold-valued data reads 
\begin{align}
\argmin_{u \in \mathcal{M}^K} \dist(\mathcal A(u),f)^p + \  \TGV_\lambda(u).
\end{align}
In a univariate setting, 
$
\TGV_\lambda(u) = 
\argmin_{v \in \mathcal{M}^K} \lambda_0 \sum_i \dist(u_{i+1},v_i) + 
\lambda_1 \sum_i D ([u_{i},v_{i}],[u_{i-1},v_{i-1}]),
$ 
where $D$ generalizes the distance between 
the ``differences'' $[u_{i},v_{i}]$ and  $[u_{i-1},v_{i-1}]$
which are here represented as tuples $[u_{i},v_{i}]$ of points $u_{i},v_{i}$ in the manifold.
One instantiation of $D$ employed in \cite{bredies2017total} uses the parallel transport in the manifold: 
$D ([u_{i},v_i],[u_{i-1}, v_{i-1}]) =  \| log_{u_{i}} v_{i} -  \pt_{u_{i-1},u_{i}} \log_{u_{i-1}} v_{i-1} \|_{u_{i}}.$
Here $\log$ denotes the inverse of the Riemannian $\exp$ mapping such that $\log_{u_{i}} v_{i}$
denotes a tangent vector sitting in $u_{i}$ ``pointing'' to $v_{i}.$ 
The symbol $\pt_{u_{i-1},u_{i}}$ denotes the parallel transport along a shortest geodesic connecting $u_{i-1}$ and $u_{i}$
such that 
$\log_{u_{i}} v_{i} -  \pt_{u_{i-1},u_{i}} \log_{u_{i-1}} v_{i-1}$ 
is a tangent vector sitting in $u_{i}$ and, 
$ \|\cdot \|_{u_{i}}$ 
denotes the norm induced by the Riemannian scalar product in the point $u_{i}.$  
The other instantiation of $D$ employed in \cite{bredies2017total} uses Schild's ladder:
$D ([u_{i},v_{i}],[u_{i-1},v_{i-1}]) = \dist(S(u_{i-1},u_{i},v_{i-1}) ,v_i),$ 
where the Schild's point $S(u_{i-1},u_{i},v_{i-1})$  $= [u_{i-1},\midp(u_{i},v_{i-1})]_2,$ 
i.e., the point obtained by evaluating the (constant speed) geodesic $\gamma$ connecting $\gamma(0)=u_{i-1}$ 
and the midpoint $\gamma(1) = \midp(u_{i},v_{i-1})$ of $u_{i}$ and $v_{i-1}$ at $t=2,$ 
that is, $S(u_{i-1},u_{i},v_{i-1}) = \gamma(2).$ The Schild variant may be seen as an approximation to the parallel variant which has the advantage of faster computation times.  
In the multivariate setup,
\begin{align}\notag
R(u) =  &\TGV_\lambda(u)  
= \min _{v^ 1,v^ 2} 
\lambda_1 \sum _{l,k}  \dist(u_{l+1,k},v^ 1_{l,k}) + \lambda_1 \sum _{l,k} \dist(u_{l,k+1},v^ 2_{l,k})\\
&+ \lambda_0 \sum_{l,k}  D\big([u_{l,k},v^1_{l,k}],[u_{l-1,k},v^1_{l-1,k}]\big) 
+ \lambda_0 \sum_{l,k} D\big([u_{l,k},v^2_{l,k}],[u_{l,k-1},v^2_{l,k-1}]\big)  \notag  \\
&  +\lambda_0 \sum_{i,j} D^{\text{sym}} ([u_{i,j},v^1_{i,j}],[u_{i,j},v^2_{i,j}],[u_{i,j-1},v^1_{i,j-1}],[u_{i-1,j},v^2_{i-1,j}]),
\label{eq:DefTGVreg} 
\end{align}
with $D$ as in the univariate setup explained above; $D^{\text{sym}}$ realizes a symmetrized gradient in the manifold setting which is defined in terms of $D.$ Due to the space consumption for providing details on $D^{\text{sym}},$ we omit the precise definitions here and refer to \cite{bredies2017total} for details.

\section{Existence Results}\label{sec:well-posedness}

In this section we derive well-posedness results of the variational problem \eqref{eq:TichManiGeneral} with indirect measurement terms. In particular, we consider the regularizers presented in Section~\ref{sec:SpecificRegularizers}.
For the $\TV$ regularizer \eqref{eq:DefTVreg} and the $V^q$ regularizers, we get the existence of minimizers without additional constraints on the measurement operator $\mathcal A$; the corresponding result is formulated as 
Theorem~\ref{thm:ExistenceMinimizers}. It 
also applies to the mixed first and second order regularizer \eqref{eq:mixedFirstSecond} (with $\mu_1 \neq 0$). 
For the $\TGV$ regularizer \eqref{eq:DefTGVreg} and the (pure) second order regularizer \eqref{eq:mixedFirstSecond} (with $\mu_1 = 0$) we get the existence of minimizers with some constraints on the measurement operator $\mathcal A$; the corresponding results are formulated as 
Theorem~\ref{thm:ExistenceMinimizersSecondOrder} and Theorem~\ref{thm:ExistenceMinimizersTGV}. 
For compact manifolds such as the spheres or the Grassmannians, however, we get the existence of minimizers 
of the $\TGV$ regularizer \eqref{eq:DefTGVreg} and the  second order TV regularizer $\TV_2$ of \eqref{eq:mixedFirstSecond} with $\mu_1 = 0$ without additional constraints on the measurement operator $\mathcal A.$ 

We first consider a general regularizer $R$ and show that, 
under certain assumptions on the regularizing term $R,$ 
the properties of the indirect measurement term \eqref{eq:DefDist4Vectors} in \eqref{eq:TichManiGeneral} yield the existence of a minimizer. We will apply this result to the $\TV$ regularizer \eqref{eq:DefTVreg} and the $V^q$ regularizers.
Later on, we will also derive a variant which applies to the second order regularizers. The precise formulation is as follows.

\begin{theorem}\label{thm:ExistenceCondR}
	We consider a sequence of signals $u^{(n)},$ and use the notation $\diam(u^{(n)})$ to denote the diameter of $u^{(n)}$ (seen as a set).
	If $R$ is a regularizing term such that $R(u^{(n)}) \to \infty,$ as $\diam(u^{(n)}) \to \infty,$ 
	and $R$ is lower semicontinuous, then the variational problem \eqref{eq:TichManiGeneral} with indirect measurement data term \eqref{eq:DefDist4Vectors} has a minimizer.
\end{theorem}

In order to prove Theorem~\ref{thm:ExistenceCondR}, we start showing some lemmata.
The first lemma is a simple consequence of the triangle inequality as for instance pointed out in \cite{afsari2011riemannian}
(for nonnegative weights); the lemma provides an extension to the situation including nonnegative weights.
\begin{lemma} \label{lem:MeansInBoundedSet}
	Let all $u_j,$ $j=1,\ldots,K,$ be contained in a ball $B(x,r)$ of radius $r$ around $x \in \mathcal M.$
	Then every mean $v^\ast \in \mean(a,u) = \argmin_{v \in \mathcal{M}} \ \sum_j  a_{j} \dist(v,u_j)^2$  in the sense of \eqref{eq:DefMatVecAnaManiIth} for a weight vector $a$ with $\sum a_j>0,$ is contained in the larger ball $B(x,R),$ with $R$ independent of the particular $u,$
	(but $R$ dependent on $r$ and $a.$)	
\end{lemma}
\begin{proof}
	Let $A^+$ be the sum of the positive weights in the weight vector $a,$
	and $A^-$ be the sum of the absolute values of the negative weights in the weight vector $a,$ respectively.
	By our assumption $A^+ > A^-,$ or, $A^0:=A^+-A^- > 0.$
	We let 
	\begin{equation}\label{eq:RatioRr}
		R := C \ r  \quad \text{ with } \quad  C := \frac{2(A^+ + A^-)}{A^0}. 
	\end{equation}
	Let $v' \not\in B(x,R).$ We estimate the functional value of $v'$ from below using $R':= \dist(v',x)$.
	 For any $u_j,$ $j=1,\ldots,K,$ 
	the reverse triangle inequality implies that
	$\dist(v',u_j) \geq R'-r,$
	and the triangle inequality implies $\dist(v',u_j) \leq R'+r.$
	Together, 
	\begin{align} \label{eq:estDistEnergyBall}
		\sum_j  a_{j} \dist(v',u_j)^2 &= 
		\sum_{j:a_j \text{ positive}} a_{j} \ \dist(v,u_j)^2  
		-  \sum_{j:a_j \text{ negative}} |a_{j}|  \ \dist(v,u_j)^2  \\
		&\geq \sum _{j:a_j \text{ positive}} a_{j} \ (R'-r)^2  - \sum_{j:a_j \text{ negative}} |a_{j}| \	 
		= A^+ (R'-r)^2 - A^{-} (R'+r)^2. \notag
	\end{align}
	We further have  
	$
	A^+ (R'-r)^2 - A^{-} (R'+r)^2 
	= R' (A^0 R'- 2 (A^+ + A^-) r) + A^0  r^2.
	$
	For the term in brackets, we have 
	\begin{align} \label{eq:termInBraketsLTzero}
	A^0 R'- 2 (A^+ + A^-) r > A^0 R - 2 (A^+ + A^-) r  =  \left(A^0  \frac{2(A^+ + A^-)}{A^0}   - 2 (A^+ + A^-)\right) r   =  0 
	\end{align} 
	by our definition of $R$ in \eqref{eq:RatioRr}.
	Applying the estimate \eqref{eq:termInBraketsLTzero} to \eqref{eq:estDistEnergyBall}, we see that   
	$\sum_j  a_{j} \dist(v',u_j)^2 > A^0  r^2.$ On the other hand,
	for the center point $x$ of the balls, we have  
	$
	\sum_j  a_{j} \dist(v',x)^2 \leq \sum_j  a_{j}  r^2 = A^0  r^2.
	$
	Hence, no point $v' \not\in B(x,R)$ can be a minimizer which shows the assertion.

\end{proof}
Note that, for nonnegative weights, the constant $C$ in \eqref{eq:RatioRr} equals two.

\begin{lemma} \label{lem:DatTermLSC}
	The data term \eqref{eq:DefDist4Vectors} is lower semicontinuous.
\end{lemma}

\begin{proof}
	In order to show that the sum in \eqref{eq:DefDist4Vectors} is a lower semicontinuous function of $u,$
	it is enough to show that each summand $u \mapsto \dist(\mathcal A(u)_i,f_i)^p$ is 
	a lower semicontinuous function on $\mathcal M^K$ for each $i=1,\ldots,N.$
	To this end, we consider a sequence of $(u^{(n)})_{n \in \mathbb N},$ each $u^{(n)} \in \mathcal M^K,$
	such that $u^{(n)} \to u$ in $ \mathcal M^K,$ as $n \to \infty.$ By the monotony of the power functions, it is sufficient to show that 	
	\begin{equation}\label{eq:ToShowLSC}
	\dist(\mathcal A(u)_i,f_i)  \leq   \liminf_{n}  \dist(\mathcal A(u^{(n)})_i,f_i), \quad \text{ for each $i=1,\ldots,N.$}
	\end{equation}
	We let $v^{(n)} \in \mathcal A(u^{(n)})_i,$ i.e., $v^{(n)}$ is a minimizer of 
	$ v \mapsto \sum_j  a_{i,j} \dist(v,u^{(n)}_j)^2,$ nearest to $f_i.$
	Similarly, let $v^\ast \in \mathcal A(u)_i$ meaning that $v^\ast$ is a minimizer of
	$ 
	v \mapsto \sum_j  a_{i,j} \dist(v,u_j)^2,
	$
	also nearest to $f_i.$
	Since, by assumption, $u^{(n)} \to u$ in $\mathcal M^K,$ we find $x \in \mathcal M$ and a positive number $r$ such 
	that all $u^{(n)}_j$ together with all $u_j,$ 
	$j=1,\ldots,K,$ 
	$n \in \mathbb N$ are contained in a common ball $B(x,r)$ around $x$ with radius $r.$ Then Lemma~\ref{lem:MeansInBoundedSet} tells us that, for every $i=1,\ldots,N$, there is a positive number $R_i$
	such that all $v^{(n)}_i,$ $n \in \mathbb N$, together with $v^\ast_i,$ lie in  $B(x,R_i).$ 
	Taking $R = \max_{i=1,\ldots,N} R_i,$ all $v^{(n)}_i,$ $n \in \mathbb N$, together with $v^\ast_i,$ lie in  $B(x,R)$ 
	for all $i=1,\ldots,N.$ Hence, the $v^{(n)}$ form a bounded sequence in  $\mathcal M^N.$
	
	Now, in view of the right hand side of \eqref{eq:ToShowLSC}, we take a subsequence $v^{(n_k)}$ of $v^{(n)}$ such that 
	\begin{equation}\label{eq:ApproxRHsF}
	\lim_{k \to \infty } \dist(v^{(n_k)}_i,f_i)  = \liminf_{n}   \dist(\mathcal A(u^{(n)})_i,f_i), \quad \text{ for all $i=1,\ldots,N.$}
	\end{equation}
	Since, by our above argumentation, the $v^{(n)}$ form a bounded sequence in  $\mathcal M^N,$
	the subsequence  $v^{(n_k)}$ is bounded as well, and  since $\mathcal M$ is geodesically complete,   
	we may extract a convergent subsequence of  $v^{(n_k)}$ which we (abusing notation for better readability) also denote by  
	 $v^{(n_k)}.$ We let 
	 \begin{equation}\label{eq:DefvPrime}
	 v' := \lim\nolimits_{k \to \infty} v^{(n_k)}.
	 \end{equation}
	 
	 Next, we show that 
	 \begin{equation}\label{eq:vPrimeIminimizes}
	 	v'_i \in  \argmin_{v \in \mathcal{M}} \ \sum\nolimits_j  a_{i,j} \dist(v,u_j)^2.	
	 \end{equation}
	 To see this, we assume that $v'_i$ were not a minimizer in \eqref{eq:vPrimeIminimizes}.
	 Then there were a minimizer $v^\ast_i$ of the sum in \eqref{eq:vPrimeIminimizes} such that 
	 $\sum_j  a_{i,j} \dist(v^\ast_i,u_j)^2 < \sum_j  a_{i,j} \dist(v'_i,u_j)^2.$ 
	 Since  $v^{(n_k)} \to v'$ and $u^{(n_k)} \to u$   as $k \to \infty,$
	  we have 
	  \begin{align}\label{eq:Eq1ShowLS}
	    \sum\nolimits_j  a_{i,j} \dist(v'_i,u_j)^2 = \lim_{k \to \infty}	\sum\nolimits_j  a_{i,j} \dist(v^{(n_k)}_i,u^{(n_k)}_j)^2.
	  \end{align}
	 Hence, for sufficiently large $k_0,$ we have that  
	   $\sum_j  a_{i,j} \dist(v^\ast_i,u^{(n_{k_0})}_j)^2 < \sum_j  a_{i,j} \dist(v^{(n_{k_0})},u^{(n_{k_0})}_j)^2.$
	 This contradicts $v^{(n_{k_0})}$ being a minimizer of  
	 the mapping $v \mapsto \sum\nolimits_j  a_{i,j} \dist(v,u^{(n_{k_0})}_j)^2,$
	 and thus shows \eqref{eq:vPrimeIminimizes}.
	 Summing up,
	 \begin{align}
	 	 \dist(\mathcal A(u)_i,f_i) \leq  \dist(v_i,f_i) =  \liminf_{n}   \dist(\mathcal A(u^{(n)})_i,f_i), 
	 \end{align}
	 for all $i.$
	 The inequality is by \eqref{eq:vPrimeIminimizes}, 
	 and the equality is by \eqref{eq:ApproxRHsF}.
	 This shows \eqref{eq:ToShowLSC} which completes the proof.	
\end{proof}

\begin{proof}[Proof of Theorem~\ref{thm:ExistenceCondR}]
	The lower semicontinuity of the data term \eqref{eq:DefDist4Vectors} is shown in Lemma~\ref{lem:DatTermLSC}.
	Together with the assumed  lower semicontinuity of the regularizing term $R,$ the functional in \eqref{eq:TichManiGeneral} is 
	lower semicontinuous. Hence, in order to show the assertion of the theorem, we show that the functional 
	\begin{align}\label{eq:FunctionalFinExProof}
	  F(u) = \dist(\mathcal A(u),f)^p + R(u),
	\end{align}
	of  \eqref{eq:TichManiGeneral} is coercive, i.e.,
	we show that there is $\sigma \in \mathcal M^K$ such that for all sequences $u^{(n)}$ in $\mathcal M^K,$
	\begin{align}\label{eq:defCoerciveMF}
	\dist(u^{(n)},\sigma) \to \infty \ \text{ as } \ n \to \infty    \qquad   
	 \text{implies } \qquad   F(u^{(n)}) \to \infty \ \text{ as } \ n \to \infty.
	\end{align}
	Note that, by the reverse triangle inequality, we may replace the $\sigma \in \mathcal M^K$ in \eqref{eq:defCoerciveMF}
	by any other  $\sigma' \in \mathcal M^K;$ in other words, if \eqref{eq:defCoerciveMF} is true for one element in $\mathcal M^K,$
	it is true for any other element in $\mathcal M^K$ as well.
	
	Towards a contradiction suppose that $F$ is not coercive. Then there is $\sigma \in \mathcal M^K$ and a sequence $u^{(n)}$ in
	 $\mathcal M^K,$ such that 
	 \begin{align}\label{eq:2contradict} 
	 \dist(u^{(n)},\sigma) \to \infty \  \text{ and } \ F(u^{(n)})  \text{ does not converge to $\infty.$} 
	 \end{align}  
	 Hence, there is a subsequence of $u^{(n)}$ with bounded value of $F,$ i.e., 
	 there is a subsequence $u^{(n_k)}$ of $u^{(n)},$ as well as a constant $C>0$ such that 
	\begin{align}\label{eq:FisBounded4Subseq} 
	   F(u^{(n_k)}) \leq C  \qquad \text{for all $k \in \mathbb N.$}
	\end{align} 
	Hence, $R(u^{(n_k)}) \leq C.$ This implies that $\diam(u^{(n_k)})$ does not converge to $\infty.$
    (Recall that, by assumption, $\diam(u^{(n)}) \to \infty$ implies that $R(u^{(n)}) \to \infty.$ )
	Therefore, there is another subsequence of $u^{(n_k)}$ which we, for the sake of readability, again denote by  
	$u^{(n_k)}$ and a constant $C'>0$ such that 
	\begin{align} \label{eq:diamBounded4allK}
	\diam(u^{(n_k)}) \leq C'  \qquad \text{for all $k \in \mathbb N.$}
	\end{align} 
	Hence,  all items of $u^{(n_k)}$ are contained in a $C'$ ball around the first element $x_k:=u^{(n_k)}_1 $ of $u^{(n_k)}$,  i.e.,
	\begin{align}    
	u^{(n_k)} \subset B(x_k,C') \quad \text{ for all } k \in \mathbb N.
	\end{align} 
	By Lemma~\ref{lem:MeansInBoundedSet} there is $C''>0$ such that all items of $\mathcal A(u^{(n_k)})$ lie in a $C''$ ball around $x_k,$ i.e.,
	\begin{align}    
	\mathcal A(u^{(n_k)}) \subset B(x_k,C'') \quad \text{ for all } k \in \mathbb N.
	\end{align} 
	Here,  we view $\mathcal A(u^{(n_k)})$ as the set of entries of the corresponding vector.
	On the other hand, by \eqref{eq:FisBounded4Subseq},
	\begin{align}
	\dist(\mathcal A(u^{(n_k)}),f)^p \leq C  \quad \text{ for all } k \in \mathbb N,
	\end{align}
	with the constant $C$ of \eqref{eq:FisBounded4Subseq}.
	This implies that there is a constant $C'''>0$ such that, for any item $\mathcal A(u^{(n_k)})_i$,
	\begin{align}
	 \mathcal A(u^{(n_k)})_i  \in  B(f_1, C''' )      \quad \text{ for all } k \in \mathbb N,  i \in \{1,\ldots,N\}.
	\end{align}
	In particular,
	\begin{align}
	 \mathcal A(u^{(n_k)})_1  \in  B(f_1, C''' ) \cap B(u^{(n_k)}_1,C'')    \quad \text{ for all } k \in \mathbb N.
	\end{align} 
	This guarantees that the right hand-sets are nonempty for any $k.$
	Hence, the sequence $u^{(n_k)}_1$ is bounded in $\mathcal M.$ 
	Hence, the sequence $u^{(n_k)}$ is bounded in $\mathcal M^K,$ since $\diam(u^{(n_k)})$ is bounded by \eqref{eq:diamBounded4allK}.
	This contradicts \eqref{eq:2contradict} according to which $\dist(u^{(n_k)},\sigma) \to \infty$ for the subsequence $u^{(n_k)}$ of $u^{(n)}$ which states that $u^{(n_k)}$ is unbounded.
	Hence, $F$ is coercive, which together with its lower semicontinuity guarantees the existence of minimizers which completes the proof.  
\end{proof}

\begin{theorem}\label{thm:ExistenceMinimizers}
	The inverse problem \eqref{eq:TichManiGeneral} for manifold-valued data with $\TV$ regularizer 
	and with the $V^q$ regularizer \eqref{eq:VqRegulizer} has a minimizer.
	The same statement  applies to the mixed first and second order regularizer \eqref{eq:mixedFirstSecond}
	with non vanishing first order term ($\mu_1 \neq 0$).
\end{theorem}

\begin{proof}
	We apply Theorem~\ref{thm:ExistenceCondR} for the mentioned regularizers.
	The $\TV$ regularizer \eqref{eq:DefTVreg} is continuous by the continuity of the Riemannian distance function. 
	Let $u^{(n)}$ be  a sequence such that  $\diam(u^{(n)}) \to \infty.$ Then, $\TV(u^{(n)}) \geq \diam(u^{(n)})$
	which implies $\TV(u^{(n)}) \to \infty.$ Hence, we may  apply Theorem~\ref{thm:ExistenceCondR} for the $\TV$ regularizer \eqref{eq:DefTVreg}.
	
	For the $V^q$ regularizer \eqref{eq:VqRegulizer} basically the same argument applies. Again, the continuity of the Riemannian distance function implies that
	the regularizer is a continuous function of the data.  
	Towards the other condition of Theorem~\ref{thm:ExistenceCondR}, 	
	let $u^{(n)}$ be  a sequence such that  $\diam(u^{(n)}) \to \infty,$ 
    and assume that there is a subsequence $u^{(n_k)}$ of $u^{(n)}$ such that $V^q(u^{(n_k)}) \leq C.$ Then,
    for this subsequence, $\diam(u^{(n_k)}) \leq C^{1/q} \cdot S$ for all $k,$ where $S$ is the sum of all dimensions of the considered image/data. This contradicts $\diam(u^{(n)}) \to \infty,$ 
    and therefore shows that  
    $V^q(u^{(n_k)}) \to \infty$ whenever $\diam(u^{(n)}) \to \infty.$
    Hence, we may  apply Theorem~\ref{thm:ExistenceCondR} for the $V^q$ regularizer \eqref{eq:VqRegulizer}.
	
	Finally, we consider the mixed first and second order regularizer \eqref{eq:mixedFirstSecond}
	with non vanishing first order term ($\mu_1 \neq 0$). By the argument for the $\TV$ regularizer above,
	$\diam(u^{(n)}) \to \infty$ implies $R(u^{(n)}) = \mu_1 \TV(u^{(n)}) + \mu_2 \TV_2(u^{(n)}) \to \infty$
	since $\mu_1 \neq 0.$ The lower semicontinuity of \eqref{eq:DefTGVreg} is a consequence of 
	the continuity of the $\TV$-term together with the lower semicontinuity of the $\TV_2$ term shown in
	Lemma~\ref{lem:TV2islsc}.
	Hence, we may  apply Theorem~\ref{thm:ExistenceCondR} for the mixed first and second order regularizer \eqref{eq:mixedFirstSecond}
	with $\mu_1 \neq 0.$
	Together, this shows the existence of minimizers for the considered regularizing terms and completes the proof.	
\end{proof}

For the previous result, and in the following, we need the lower semicontinuity of the $\TV_2$ regularizer which is stated in the next lemma.

\begin{lemma}\label{lem:TV2islsc}
	The second order $\TV_2$ regularizer, i.e., the regularizer in \eqref{eq:mixedFirstSecond}
	with $\mu_1 = 0 ,$ is lower semicontinuous.
\end{lemma}

To streamline the presentation, the proof of Lemma~\ref{lem:TV2islsc} is given in Appendix~\ref{sec:Appendix}. 

We note that Theorem~\ref{thm:ExistenceCondR} does not apply to the $\TGV$ regularizer \eqref{eq:DefTGVreg},
since, for a general imaging operator $\mathcal A,$ $\diam(u^{(n)}) \to \infty$ does not imply 
$\TGV_\lambda(u^{(n)}) \to \infty$
as the following univariate toy example shows.
\begin{example}
	let $\gamma$ be a unit speed geodesic im $\mathcal M$ and consider the univariate signals $u^{(n)}$ be given by   
	$u^{(n)}_j = \gamma (nj),$ $j=-1,0,1.$ Then, $\TGV_\lambda(u^{(n)}) = 0$ for all $n \in \mathbb N,$ but $\diam(u^{(n)}) \to \infty.$
	The same applies to second order $\TV_2$ regularizer \eqref{eq:mixedFirstSecond}
	with $\mu_1 = 0$, i.e., we have $\TV_2(u^{(n)}) = 0.$
	Further consider the imaging operator $\mathcal A$ given by the matrix $A= (1/2,0,1/2)$
	having only one row,
	and the measurement $f:=\gamma(0) \in \mathcal M.$
	Then, the functional in \eqref{eq:TichManiGeneral}  with $R = \TGV_2$ or $R=\TV_2$ and imaging operator $\mathcal A$
	is not coercive.
\end{example}
This example also shows that, in general, the functional in \eqref{eq:TichManiGeneral}  with $R = \TGV,$ the  
$\TGV$ regularizer of \eqref{eq:DefTGVreg}, is not coercive. 
The same statement also applies to the $\TV_2$ regularizer \eqref{eq:mixedFirstSecond}
with $\mu_1 = 0,$ i.e., the functional in \eqref{eq:TichManiGeneral} with $R=\TV_2$ is neither coercive nor does  
$\TV_2$ fulfill the condition  
$\TV_2 \to \infty$ as $\diam(u^{(n)}) \to \infty.$ 
To account for that, we give a variant of Theorem~\ref{thm:ExistenceCondR} which applies to these regularizers. 
This comes at the cost of additional constraints to $\mathcal A.$
 
\begin{theorem}\label{thm:ExistenceCondR4SecondOrd}
	Let $(l_0,r_0),$ $\ldots,$ $(l_S,r_S)$ be $S$ pairs of (a priori fixed) indices.
	We assume that $R$ is lower semicontinuous.
	We further assume that $R$ is a regularizing term such that, 
	for any sequences of signals $u^{(n)},$  
	the conditions $\diam(u^{(n)}) \to \infty$ and $\dist(u^{(n)}_{l_s},u^{(n)}_{r_s}) \leq C,$ for some $C>0$ and for all $n \in \mathbb N$ and all $s \in \{0,\ldots,S\},$
	imply that $R(u^{(n)}) \to \infty.$  
	If $\mathcal A$ is an imaging operator such that there is a constant $C'>0$ such that, for any signal $u$, $dist(u_{l_s},u_{r_s}) \leq C' \max (\diam (\mathcal A u),R(u)),$ for all $s \in \{0,\ldots,S\}$,
	 then the variational problem \eqref{eq:TichManiGeneral} with indirect measurement data term \eqref{eq:DefDist4Vectors} has a minimizer.
\end{theorem} 
 
The proof of Theorem~\ref{thm:ExistenceCondR4SecondOrd} uses modifications of the techniques used to show Theorem~\ref{thm:ExistenceCondR} and is given in Appendix~\ref{sec:Appendix}.

We next apply Theorem~\ref{thm:ExistenceCondR4SecondOrd} to the (pure) second order $\TV_2$ regularizer.

\begin{theorem}\label{thm:ExistenceMinimizersSecondOrder}
	Consider the inverse problem \eqref{eq:TichManiGeneral} for manifold-valued data for the (pure) second order $\TV_2$ regularizer \eqref{eq:mixedFirstSecond} with $\mu_1 = 0$.
	For the univariate situation, we assume that the imaging operator $\mathcal A$ has the property that there 
	is an index $j_0$ and a constant $C>0$ such that, for any signal $u$, 
	\begin{align}\label{eq:AddAssumptionExUni}
		dist(u_{j_0},u_{j_0+1}) \leq C \max (\diam (\mathcal A u),\TV_2(u)).
	\end{align}
	For the bivariate situation, we assume that the imaging operator $\mathcal A$ has the property that there 
	are two indices $j_0=(x_0,y_0)$ and $j_1=(x_1,y_1)$  and a constant $C>0$ such that, for any signal $u$, 
	\begin{align}\label{eq:AddAssumptionExMulti}
		\max \left(\dist(u_{x_0,y_0},u_{x_0+1,y_0}),\dist(u_{x_1,y_1},u_{x_1,y_1+1})\right) \leq C \max (\diam (\mathcal A u),\TV_2(u)),\quad  
	\end{align}
	Then, the inverse problem \eqref{eq:TichManiGeneral} for manifold-valued data with the second order $\TV_2$ regularizer given by \eqref{eq:mixedFirstSecond}	for $\mu_1 = 0$ has a minimizer.
\end{theorem}

The proof of Theorem~\ref{thm:ExistenceMinimizersSecondOrder} is given in Appendix~\ref{sec:Appendix}.

\begin{theorem}\label{thm:ExistenceMinimizersTGV}
	Consider the inverse problem \eqref{eq:TichManiGeneral} for manifold-valued data for the $\TGV$ regularizer \eqref{eq:DefTGVreg}.
	For the univariate situation, we assume that the imaging operator $\mathcal A$ has the property that there 
	is an index $j_0$ and a constant $C>0$ such that, for any signal $u$, \eqref{eq:AddAssumptionExUni} is fulfilled with $\TV_2$ replaced by $\TGV_\lambda$.
	For the bivariate situation, we assume that the imaging operator $\mathcal A$ has the property that there 
	are two indices $j_0=(x_0,y_0)$ and $j_1=(x_1,y_1)$  and a constant $C>0$ such that, for any signal $u$, 
	\eqref{eq:AddAssumptionExMulti} is fulfilled with $\TV_2$ replaced by $\TGV_\lambda$, and that in addition, there is a further index $j_2=(x_2,y_2)$
	with $|y_2-y_0|=1$ (neighboring lines) such that 
	\begin{align}\label{eq:4fromCrossToAll}
	\dist(u_{x_2,y_2},u_{x_2+1,y_2}) \leq C \max (\diam (\mathcal A u),\TGV_\lambda(u)). 
	\end{align}
	(In the latter condition the index $j_2$ can be replaced by 
    $j_2=(x_2,y_2)$ with $|x_2-x_1|=1,$ and 
	$
		\dist(u_{x_2,y_2},u_{x_2,y_2+1}) \leq C \max (\diam (\mathcal A u),\TGV_\lambda(u))
	$.)
	Then, the inverse problem \eqref{eq:TichManiGeneral} for manifold-valued data with $\TGV$ regularizer \eqref{eq:DefTGVreg} has a minimizer.
\end{theorem}

The proof of Theorem~\ref{thm:ExistenceMinimizersTGV} is given in Appendix~\ref{sec:Appendix}.

\begin{cor}\label{cor:ExistenceMinimizersSecondOrderCompact}
	Let $\mathcal M$ be a compact manifold. 
	The inverse problem \eqref{eq:TichManiGeneral} for data living in $\mathcal M$ with $\TGV$ regularizer \eqref{eq:DefTGVreg} has a minimizer.
	The same statement applies to the (pure) second order $\TV_2$ regularizer \eqref{eq:mixedFirstSecond}
	with $\mu_1 = 0$.
\end{cor}

\begin{proof}
	Noting that both regularizers are lower semicontinuous (cf. the proofs of  Theorem~\ref{thm:ExistenceMinimizersSecondOrder} and 
	 Theorem~\ref{thm:ExistenceMinimizersTGV}),
	the statement is an immediate consequence of Theorem~\ref{thm:ExistenceCondR4SecondOrd} by noting that the second condition 
	on $R$ is trivially fulfilled since  $\mathcal M$ is bounded, or directly by the fact that the compactness of  
    $\mathcal M$ implies the coercivity of $R$. 
\end{proof}

\section{Algorithms}\label{sec:algorithm}

In the following we derive algorithms for the manifold valued analogue \eqref{eq:TichManiGeneral} of Tichanov-Phillips regularization. 
We consider extensions of generalized forward-backward schemes for data terms in \eqref{eq:TichManiGeneral} with $p>1.$ 
More precisely, we propose a variant based on a trajectory method together with a Gau\ss-Seidel type 
update strategy as well as a stochastic variant of the generalized forward-backward scheme.
Further, we consider proximal point algorithms which may also be used in the case $p=1.$ 
In this paper, we explicitly consider regularizers which we can decompose into basic atoms for which we in turn can compute their proximal mappings. Examples of such regularizers are the ones discussed in Section~\ref{sec:SpecificRegularizers}, e.g., $\TV$ and $\TGV$ regularizers.

\subsection{Basic Algorithmic Structures}\label{sec:Basic_Algorithmic_Structures}

\paragraph{Basic Generalized Forward Backward Scheme.} 
In \cite{baust2016combined}, we have proposed a generalized forward backward algorithm for DTI data with a voxel-wise indirect data term. We briefly recall the basic idea and then continue with a further developement of this scheme. 

We denote the functional in \eqref{eq:DefTivchManiTV} by $\mathcal F$ and decompose it into the data term $\mathcal D$ and the regularizer $R$ which we further decompose into atoms $R_k,$ i.e.,
\begin{align}
\mathcal F(u) = \mathcal D (u) +  \lambda \ R(u) =  \mathcal D (u) +  \lambda \sum\nolimits_{k=1}^{K'}  R_k(u)
\end{align}
with 
\begin{align}\label{eq:DataTermAsD}
  \mathcal D (u)   = \dist(\mathcal A(u),f)^p. 	
\end{align}
The basic idea of a generalized forward-backward scheme is to perform a gradient step for the explicit term, here $\mathcal D$,  as well as proximal mapping step for each atom of the implicit term, here $R_i.$ 
Here, the  proximal mapping \cite{moreau1962fonctions, ferreira2002proximal, azagra2005proximal} of a function $g$ on a manifold  $\mathcal{M}'$ is given by
\begin{align} \label{eq:prox_mapping_abstract}
\prox_{\mu g} x = \argmin_y g(y) + \tfrac{1}{2 \mu} \dist(x,y)^2, \qquad \mu > 0.  
\end{align}
For general manifolds, the proximal mappings \eqref{eq:prox_mapping_abstract} are not globally defined, and the minimizers are not unique, at least for general possibly far apart points; cf. \cite{ferreira2002proximal, azagra2005proximal}.
This is a general issue in the context of manifolds that are -- in a certain sense -- a local concept
involving objects that are often only locally well defined. In case of ambiguities, we hence consider the above objects as set-valued quantities. Furthermore, often the considered functionals are not convex; hence the convergence to a globally optimal solution cannot be ensured. Nevertheless, as will be seen in the numerical experiments section, we experience a good convergence behavior in practice.
This was also observed in previous works such as  \cite{bergmann2014second,bavcak2016second} where the involved manifold valued functionals are not convex either.

Concerning the gradient of the data term \eqref{eq:DataTermAsD}, we write 
\begin{align}\label{eq:DataTermAsDi}
\mathcal D (u)   = \sum\nolimits_{i=1}^N   \mathcal D_i (u),  \quad \text{ with }\quad   \mathcal D_i (u):=   \dist(\mathcal A(u)_i,f_i)^p 	
\end{align}
with $ p \in (1,\infty).$
Then, we have, for the gradient of $\mathcal D$ w.r.t.\ the variable $u_l,$
\begin{align}
	\nabla_{u_l} \mathcal D (u) =    \sum\nolimits_{i=1}^N  \nabla_{u_l} \mathcal D_i (u).
\end{align}  
The gradient of $\mathcal D_i$ w.r.t.\ $u_l$ is given in Theorem~\ref{thm:ComputeDerivativeOfM}.
Then the overall algorithm reads 
\begin{align}\notag
 \text{ Iterate} & \text{ w.r.t.\ $n$ :} \\
 1. &\text{ Compute } u^{(n+0.5)} =  \exp_{u^{(n)}} \left(-\mu_n \nabla_{u} \mathcal D (u^{(n)})\right),  \qquad \quad \text{ for all $l$ }; \label{eq:algFBclassic} \\  
 2.  &\text{ Compute } u^{(n+0.5+ k/2{K'})} =  \prox_{\mu_n \lambda R_k} u^{(n+0.5+ (k-1)/2{K'})}, \qquad  \text{ for all $k.$} \quad \quad \notag 
\end{align}  
Note that, for the explicit gradient descend part, we use the $k$th iterate for all $l$ updates which corresponds to the Jacobi type update scheme.  
During the iteration, the positive parameter $\mu_n$ is decreased in a way such that $\sum_n \mu_n = \infty$ and  
such that $\sum_n \mu_n^2 < \infty.$ 
We note that for the regularizers $R$ of Section~\ref{sec:SpecificRegularizers} 
the step 2 in \eqref{eq:algFBclassic} can be massively parallelized as explained in the corresponding papers.

One issue here is the parameter choice for the gradient descent which is a common problem when dealing with gradient descent schemes.
Often some step size control, e.g., using line search techniques are employed. 
In our situation, two particular issues arise: 
(i) One particular $D_i$ may cause a low step size whereas the other $D_j$ allow for much larger steps.
Then we have to employ the small step size for all $D_j$ as well. 
(ii) In order not to touch on the balancing between the data and the regularizing term, the resulting step size parameter 
also influences the proximal mapping (realized via $\mu_n$ in \eqref{eq:algFBclassic} above.) In particular, a small step size 
in the gradient step further implies a small step size in the proximal step.
Summing up, a small step size within an atom of the data term results in a small step size for the whole loop of the iteration.

\paragraph{A Generalized Forward Backward Scheme using a Gau\ss-Seidel Type Gradient Update and a Trajectory Method.}  In order to overcome the step size issue discussed at the end of the previous paragraph, we propose to employ a Gauss-Seidel type update scheme.  More precisely, we consider the algorithm
\begin{align}\notag
\text{ Iterate} & \text{ w.r.t.\ $n$ :} \\
1. &\text{ Compute } u^{(n+i/2N)} =  
\exp_{u^{(n + (i-1)/2N)}} \left(-\mu_n \nabla_{u} \mathcal D_i (u^{(n + (i-1)/2N)})\right),  
\qquad  \text{ for all $i$ }; \label{eq:algFBwithGSupdate} \\  
2.  &\text{ Compute } u^{(n+0.5+ k/2{K'})} =  \prox_{\mu_n \lambda R_k} u^{(n+0.5+ (k-1)/2{K'})}, \qquad  \text{ for all $k.$} \qquad \qquad \notag 
\end{align}  
This has the advantage that the computation of the gradients can be performed atom wise for each $\mathcal D_i$.
In particular, if we face a small step size for a particular $\mathcal D_{i'},$ we can overcome the problem that we have to decrease both the step size of the other $\mathcal D_i$'s and the step size for the proximal mappings of the atoms $R_i$:
instead of decreasing the step size, we may use a 
{\em trajectory method}.	
Inspired by solving initial value problems for ODEs, we propose to do the following to get   
$x = u^{(n+i/2N)}$ given $x_0 = u^{(n + (i-1)/2N)}$:   
\begin{align}\notag
\text{Iterate w}&\text{.r.t.\ $r$ until $\tau \geq 1$ :} \\
  & x_r :=   \exp_{x_{r-1}} \left(- \tau_{r-1}  \mu_n \nabla \mathcal D_i(x_{r-1})  \label{eq:TrajMeth} \right)\\
  & \tau :=  \sum\nolimits_{l=0}^{r-1} \tau_{l} \notag \\
 x=   \exp &_{x_{r-1}}  \left(  \left(1 - \sum\nolimits_{l=0}^{r-2} \tau_{l}\right)  \mu_n \nabla \mathcal D_i(x_{r-1}) \right) \notag
 \end{align}  
Here, $\tau_{r-1}$ is a predicted step size for the gradient step at $x_{r-1}.$
This means instead of using a straight line we follow a polygonal path normalized by evaluating it at ``time'' $\tau=1$.
Replacing the number one by a number smaller than one would result in a damping factor w.r.t.~the explicit part of the functional. 
Note that this approach is not possible for the classical Jacobi-type update scheme.
Summing up, we propose the algorithm
\begin{align}\notag
\text{ Iterate} & \text{ w.r.t.\ $n$ :} \\
1. &\text{ Compute } u^{(n+i/2N)} =  \traj \mathcal D_i \left(u^{(n + (i-1)/2N)}\right) \qquad  \text{ for all $i$ }; \label{eq:algFBwithGSupdateAndTraj} \\  
2.  &\text{ Compute } u^{(n+0.5+ k/2{K'})} =  \prox_{\mu_n \lambda R_k} u^{(n+0.5+ (k-1)/2{K'})}, \qquad  \text{ for all $k.$} \qquad \qquad \notag 
\end{align}  
where $\traj \mathcal D_i(\cdot)$ denotes the application of the trajectory method defined by \eqref{eq:TrajMeth}.
We note that the algorithm \eqref{eq:algFBwithGSupdateAndTraj} is in many situations parallelizable with minor modifications.
For instance if $\mathcal A$ denotes the manifold analogue of a convolution operator with finite support. If the supports of the corresponding masks at two different spatial points do not overlap, the respective  $\traj \mathcal D_i$ may be computed in parallel. The minor modification in the algorithm then consists of another order of applying the operation $\traj \mathcal D_i.$

\paragraph{A Stochastic Generalized Forward Backward Scheme.}

The above algorithm given by \eqref{eq:algFBwithGSupdateAndTraj} is in a certain sense related to stochastic gradient descent schemes
(which are presently very popular for optimizing neural networks for machine learning tasks).
The update of the gradients of the mappings $\mathcal D_i$
using a Gauss-Seidel scheme type scheme may be seen as some steps of a stochastic gradient descent (with prescribed order, however).
We propose to chose the order randomly also including the proximal mapping atoms. This results in the following algorithm:
\begin{align}\notag
\text{ Iterate} & \text{ w.r.t.\ $n$ :} \\
1. & \text{ choose a random permutation $\rho$ of $\{1, 2, \ldots, N+K'\}$} \notag \\
2. & \text{ for }i = 1, 2 \ldots, N + K'  \notag\\
&\qquad\text{ if $\rho(i) \leq N:$} \qquad\text{ compute } u^{(n+i/(N+K'))} =  \traj \mathcal D_{\rho(i)} \left(u^{(n+(i-1)/(N+K'))}\right),  \label{eq:algFBwithGSupdateAndTrajRandom}  \\  
  &\qquad\text{ else:} \qquad\quad\qquad\text{ compute } u^{(n+i/(N+K'))} =  \prox_{\mu_n \lambda R_{\rho(i)-N}} u^{(n+(i-1)/(N+K'))}.\notag 
\end{align}  
A further variant of the proposed scheme would be to consider larger selections of atoms at a step.

\paragraph{The Cyclic Proximal Point Scheme.}

A reference for cyclic proximal point algorithms in vector spaces is \cite{Bertsekas2011in}.
In the context of Hadamard spaces, the concept of CPPAs was developed by \cite{bavcak2013computing}, where it is used to compute means and medians.
In the context of variational regularization methods for nonlinear, manifold-valued data,
they were first used in \cite{weinmann2014total} and then later in various variants in 
\cite{bergmann2014second,bavcak2016second,bredies2017total}. 

We briefly explain the basic principle. The idea of CPPAs is to compose the target functional into 
basic atoms and then to compute the proximal mappings of each of the atoms in a cyclic, iterative way.
In the notation used above, we have the algorithm
\begin{align}\notag
\text{ Iterate} & \text{ w.r.t.\ $n$ :} \\
1. &\text{ Compute } u^{(n+i/2N)} =  \prox_{\mu_n \mathcal D_i} u^{(n + (i-1)/2N)}\qquad  \text{ for all $i$ }; \label{eq:algCPPA} \\  
2.  &\text{ Compute } u^{(n+0.5+ k/2K')} =  \prox_{\mu_n \lambda R_k} u^{(n+0.5+ (k-1)/2K')}, \qquad  \text{ for all $k.$} \qquad \qquad \notag 
\end{align}  
We note that we in particular use \eqref{eq:algCPPA} when the exponent $p$ of the power of the distance in the data term equals $one.$ As above, the parameters $\mu_n$ are chosen such that $\sum_n \mu_n = \infty$ and such that $\sum_n \mu_n^2 < \infty.$

\subsection{Derivation of the Gradients and the Proximal Mappings for the Data Terms with Indirect Measurements.}

We here derive the gradients of the data terms $\mathcal D$ of the form \eqref{eq:DataTermAsD} which incorporate indirect measurement terms in the manifold setting
and its atoms $\mathcal D_i,$ respectively.
We further derive the proximal mappings of the atoms $\mathcal D_i.$ 

\paragraph{Gradients for the Data Term with Indirect Measurements.}

Since the gradient of $\mathcal D$ given by \eqref{eq:DataTermAsD}
equals the sum of the gradients of its atoms $\mathcal D_i,$
i.e.,
\begin{align}
\nabla \mathcal D(u)  = \sum\nolimits_i \nabla \mathcal D_i(u),
\end{align}
it is sufficient to explain how to compute the gradients of the $\mathcal D_i$.
Since 
\begin{align*}
\mathcal D_i (u) = \dist(\mathcal A(u)_i,f_i)^p 
\end{align*}
the gradient $\nabla \mathcal D_i(u)$ of $\mathcal D_i$ at $u \in \mathcal M^K$ equals the gradient of the 
distance mapping 
\begin{align}\label{eq:defDistFixSecond}
y \mapsto \dist^p_{f_i}(y) = \dist (y,f_i)^p
\end{align}
applied to the adjoint of the differential of the mapping
\begin{align}\label{eq:MapToDeriveGrad}
  M: \mathcal M^K \to \mathcal M, \qquad    u \mapsto  \argmin_{x \in \mathcal M}  \ \sum\nolimits_j  A_{i,j} \ \dist(x,u_j)^2. 
\end{align}
This is a consequence of the chain rule together with the fact that forming the adjoint of the concatenation of two differentials changes the order of the involved arguments. 
(For further background information on Riemannian geometry, we refer to the books \cite{spivak1975differential,do1992riemannian}.)
The gradient of the distance mapping $\dist^p_{f_i}$ is given by (cf., e.g., \cite{afsari2011riemannian})
\begin{align}\label{eq:gradientpPowerDist}
	\nabla\dist^p_{f_i}(y) = - p\|\exp^{-1}_y(f_i)\|^{p-2} \exp^{-1}_y(f_i).
\end{align}
Here $\exp^{-1}$ denotes the inverse of the Riemannian exponential function. 
In particular, for $p=1,2$ we have
\begin{align}\label{eq:gradient12Dist}
\nabla\dist^2_{f_i}(y) = - 2 \exp^{-1}_y(f_i),\qquad  
\nabla\dist_{f_i}(y) = - \frac{\exp^{-1}_y(f_i)}{\|\exp^{-1}_y(f_i)\|},
\end{align}
whenever $y \neq f_i.$ 

In order to find the differential of the mapping $M$ of \eqref{eq:MapToDeriveGrad},
it is sufficient to compute the differential of $M$ w.r.t.\ the elements $u_i$ in $\mathcal M.$
To this end, we use the function $\mathcal W: \mathcal M ^{K+1} \to \mathcal{TM},$
\begin{align}\label{eq:VecFieldPure}
	\mathcal W (u_1,\ldots,u_K,m) = \sum\nolimits_{j=1}^K A_{ij}  \exp^{-1}_m(u_j). 	
\end{align}
Here, $\mathcal{TM}$ denotes the tangent bundle of $\mathcal M.$
As a mapping of the argument $m,$ $\mathcal W$ is a tangent vector field.
Furthermore, an intrinsic mean is a zero of the corresponding vector field;
see \cite{groisser2004newton}. More precisely, every mean $M(u_1,\ldots,u_K)$ of $u_1,\ldots,u_K,$  the mean given by \eqref{eq:MapToDeriveGrad}, is a zero of $\mathcal W,$ i.e.,
\begin{align}\label{eq:ZeroVecField}
\mathcal W (u_1,\ldots,u_K,M(u_1,\ldots,u_K)) = 0. 	
\end{align} 
(If there is a unique zero of $\mathcal W,$ then the last component $m$ is the mean $M(u_1,\ldots,u_K)$ of $u_1,\ldots,u_K.$  In case of non-uniqueness, the characterization holds at least locally.)

Next, we calculate the derivative of the left-hand side of \eqref{eq:ZeroVecField}, i.e., of the  mapping   
\begin{align}\label{eq:ZeroVecFieldDerRightHs}
\mathcal W' : (u_1,\ldots,u_K) \mapsto \mathcal W (u_1,\ldots,u_K,M(u_1,\ldots,u_K)).
\end{align} 
We note that $\mathcal W'$ maps $\mathcal M^K$ to $\mathcal{TM},$ and so its differential maps $\mathcal{TM}^K$  to 
the second tangential bundle $\mathcal{TTM}.$ For a point $(p,v),$ or short $v,$ in $\mathcal{TM}$ 
with corresponding base point $p \in \mathcal M,$ we identify $\mathcal{TTM}_v$ with the direct sum   
$\mathcal{TM}_p \oplus \mathcal{TM}_p,$ the vertical and the horizontal subspace induced by the Riemannian connection as, e.g., discussed in detail in \cite{karcher1977riemannian}.
Let us briefly explain this identification. 
 We represent an element in $\mathcal{TTM},$ i.e., 
 a tangent vector $v'$ in a point $(p,v) \in \mathcal{TM}$, 
 by the corresponding equivalence class $v'=[\gamma]$ of    
 curves in $\mathcal{TM}$ passing through $(p,v) \in \mathcal{TM}.$ 
 Then the horizontal component $[\pi \circ \gamma]$ of $v'$ in $(p,v)$
 is given as the (equivalence class of) the tangent vector in $p$ represented by the curve $\pi \circ \gamma$ 
 obtained by projecting the representative $\gamma$ to $\mathcal M$ via 
 $\pi: \mathcal{TM} \to \mathcal M,$ $\pi(p,v)=p.$
 The vertical component of $v'$ is given by the covariant derivative of the curve $\gamma$ in $\mathcal{TM}$  
w.r.t.\ to its base curve $\pi \circ \gamma$ in $\mathcal{M}.$
We first observe that  
the horizontal component of the differential $\partial \mathcal W'$ of the mapping $\mathcal W'$ in \eqref{eq:ZeroVecFieldDerRightHs} is given by the differential of $M.$
To see this, we note that the projection of the mapping $\mathcal W'$ just equals the mean mapping $M,$
	i.e., $\pi \circ \mathcal W' = M.$
We next consider the vertical component of the differential $\partial_u \mathcal W',$ in particular
the vertical component $\partial^v_{u_{j_0}} \mathcal W'$ w.r.t.\ the variable $u_{j_0} \in \mathcal M,$ $j_0 \in \{1,\ldots,K\}.$  
By \eqref{eq:ZeroVecFieldDerRightHs} $\partial^v_{u_{j_0}} \mathcal W'$ equals the sum of the 
vertical component of the differential $\partial^v_{u_{j_0}} \mathcal W$ given by \eqref{eq:VecFieldPure} w.r.t.\ the $j_0th$ variable (not to confuse with $\mathcal W'$) and, as the second summand, the vertical component of the differential of the 
concatenation of $f_1: u_{j_0} \mapsto M(u_1,\ldots,u_K)$ and $W^m: m \mapsto W (u_1,\ldots,u_K,m).$ 
The mapping $W^m$ is a vector field, and so $W^m \circ f_1$ is a vector field along $f_1.$  
Applying the chain rule for covariant derivatives (cf., e.g., \cite{ballmann2002vector}), 
the vertical component of the differential of $W^m \circ f_1$ is given as the covariant derivative $\nabla_{\partial_{u_{j_0}}f_1}W^m.$ 
Summing up, we have for the vertical component of the differential $\partial_u \mathcal W'$ at the point $u \in \mathcal M^K$ in direction $w,$
\begin{align}\label{eq:ChainRuleAppl2VecField}
\partial^v_{u_{j_0}} \mathcal W'(u) \ w  =  
\partial^v_{u_{j_0}} \mathcal W(u) \ w + 
\nabla_{\partial_{u_{j_0}}M(u) \ w} \mathcal W^m.     
\end{align}

We first consider the vertical component of the differential $\partial^v_{u_{j_0}} \mathcal W$ 
of $\mathcal W$ w.r.t. its components $u_{j_0}$. The differential $\partial_{u_{j_0}} \mathcal W$ 
equals the differential of the mapping 
\begin{align}
u_{j_0}  \mapsto \sum_{j=1}^K A_{ij}  \exp^{-1}_m(u_j),  \qquad {j_0} \in \{1,\ldots,K\}
\end{align}
which is given by
\begin{align}  
\partial_{u_{j_0}} \mathcal W = A_{i{j_0}} \partial_{u_{j_0}} \exp^{-1}_m(u_{j_0}).
\end{align}    
Note that since $m$ is fixed, covariant derivation amounts to taking ordinary derivatives in $\mathcal{TM}_p$ and the horizontal component of the differential equals $0.$ 
The vertical component $\partial^v_{u_{j_0}} \exp^{-1}_m(u_{j_0})$ of $\partial_{u_{j_0}} \exp^{-1}_m(u_{j_0})$ can be described in terms of Jacobi fields along the geodesic 
$\gamma$ connecting $m=\gamma(0)$ and $u_{j_0}=\gamma(1).$
It is given by the boundary to initial value mapping 
\begin{align}\label{eq:Littlerjay}
	r_{j_0} := \partial^v_{u_{j_0}} \exp^{-1}_m(u_{j_0}):\quad    \mathcal{TM}_{u_{j_0}} \to \mathcal{TM}_{m}, \quad J(1) \mapsto \frac{D}{dt} J (0),     
\end{align}
where the $J$ are the Jacobi fields with $J(0)=0$ which parametrize $\mathcal{TM}_{u_{j_0}}$ via evaluation $J \mapsto J(1).$  
We note that this mapping is well-defined for non conjugate points which is the case for close enough points.
Hence, the vertical component  $\partial^v_{u_{j_0}} \mathcal W$ applied to a tangent vector $w_{j_0}$ in $\mathcal{TM}_{u_{j_0}}$
is given by 
\begin{align}\label{eq:DefRj0}
 R_{j_0} w_{j_0}  := \partial^v_{u_{j_0}} \mathcal W  w_{j_0}  = A_{i{j_0}} 	r_{j_0} \ w_{j_0}.    
\end{align}

For the second summand in \eqref{eq:ChainRuleAppl2VecField}, we first calculate the Riemannian connection $\nabla_v \mathcal W^m$ applied to the vector field $\mathcal W^m$ and an arbitrary vector $v.$ 
For  $\mathcal W^m,$ which is given by 
\begin{align}
m  \mapsto \sum\nolimits_{j=1}^K A_{ij}  \exp^{-1}_m(u_j),  
\end{align}
we first observe that 
\begin{align} \label{eq:WderivedwrtM} 
\partial_{m} \mathcal W^m = \sum\nolimits_{j=1}^K A_{i{j}} \partial_{m} \exp^{-1}_m(u_{j_0}).
\end{align}   
Its horizontal component equals the identity in every point;  
its vertical component $\partial^v_{m} \mathcal W$ applied to $v \in \mathcal{TM}_m$ equals
the covariant derivative $\nabla_v \mathcal W^m.$
Also $\nabla_v \mathcal W^m$ can be described in terms of Jacobi fields using \eqref{eq:WderivedwrtM}.
To this end, let $\gamma_j$ be the geodesic 
connecting $m=\gamma_j (0)$ and $u_{j}=\gamma_j(1),$ for $j \in \{1,\ldots,K\}.$
For each geodesic $\gamma_j,$ $j \in \{1,\ldots,K\},$ consider the Jacobi fields $J_j$ with $J_j(1) = 0$
(which correspond to geodesic variations leaving $u_j$ fixed,)
and the mappings
\begin{align}\label{eq:littlelJay}
l_j:\    \mathcal{TM}_{m} \to \mathcal{TM}_{m}, \quad J_j(0) \mapsto \frac{D}{dt} J_j (0),  \qquad    j \in \{1,\ldots,K\},
\end{align}
where the $J_j$ are the Jacobi fields with $J_j(1)=0$ which all parametrize $\mathcal{TM}_{m}$ via the evaluation map $J_j \mapsto J_j(0).$  
Again, we note that this mapping is well-defined for non conjugate points which is the case for close enough points.
Since $l_j$ equals the vertical part of $\partial_{m} \exp^{-1}_m(u_{j_0}),$ we have that 
\begin{align}\label{eq:DefL}
  L \ v := \nabla_v \mathcal W^m      = \sum\nolimits_{j=1}^K A_{i{j}} \ l_j \ v.
\end{align}
We note that the mapping $L$ is invertible whenever the mean is unique 
(since then the vector field has only one zero;
this is guaranteed, if the points $u_j$ are close enough; see \cite{afsari2011riemannian}.)
Using these derivations we may use \eqref{eq:ChainRuleAppl2VecField} to compute the derivative of the mapping $M$ mapping the 
$K$ points $u_j$ to their intrinsic mean $M(u_1,\ldots,u_K).$

\begin{theorem}\label{thm:ComputeDerivativeOfM}
Using the notation introduced above, we have 	
\begin{align}\label{eq:thmEq1}
\partial^v_{u_{j_0}} \mathcal W(u) \ w + 
\nabla_{\partial_{u_{j_0}}M(u) \ w} \mathcal W^m =  0.	
\end{align}	 
In particular, the derivative of the intrinsic mean mapping in direction $w \in \mathcal{TM}_{u_{j_0}}$ is given by  
\begin{align}\label{eq:thmEq2}
\partial_{u_{j_0}}M(u) \ w  =  - L^{-1} R_{j_0} w
\end{align}	 
where the linear mapping $L$ is given by \eqref{eq:DefL},
and the linear mappings $R_{j_0},$ $j_0 \in\{1,\ldots,K\}$,  are given by \eqref{eq:DefRj0}. 
Then, the gradient $\nabla \mathcal D_i(u)$ at $u \in \mathcal M^K$ is given 
by the  the adjoint of $- L^{-1} R_{j_0},$ applied to the gradient of the 
distance mapping given in \eqref{eq:gradientpPowerDist},
 i.e.,
\begin{align}\label{eq:thmEq3}
  \nabla_{u_{j_0}} \mathcal D_i(u) = p\|\exp^{-1}_y(f_i)\|^{p-2}  R_{j_0}^\ast   L^{-1 \ast}  \exp^{-1}_y(f_i)	
\end{align}
where $R_{j_0}^\ast$ denotes the adjoint of $R_{j_0},$ 
and $L^{-1\ast}$ denotes the adjoint of $L^{-1}.$	 
\end{theorem}

\begin{proof}
   To see \eqref{eq:thmEq1} we have a look at \eqref{eq:ZeroVecField}.
   According to \eqref{eq:ChainRuleAppl2VecField}, the vertical component of the differential of the left-hand side of \eqref{eq:ZeroVecField} is given by 
   the left-hand side of \eqref{eq:thmEq1}. The right-hand side of \eqref{eq:ZeroVecField} is the zero vector in the tangential space 
   of $M(u_1,\ldots,u_K).$ Hence, the vertical component of the differential of the right-hand side of \eqref{eq:ZeroVecField}
   equals the zero-vector in the point $M(u_1,\ldots,u_K)$ which shows \eqref{eq:thmEq1}.
   Now, \eqref{eq:thmEq2} is a consequence of the discussion near \eqref{eq:DefRj0} and \eqref{eq:DefL} above.
   The last assertion \eqref{eq:thmEq2} follows from the discussion  
   near \eqref{eq:defDistFixSecond} and \eqref{eq:MapToDeriveGrad} and by \eqref{eq:gradient12Dist}.
   
\end{proof}

If the manifold $\mathcal M$ is a Riemannian symmetric space, the Jacobi fields 
needed to compute the mappings $R_{j_0},$ 
and $L$ in the above theorem can be made more explicit.
As a reference on symmetric spaces we refer to \cite{cheeger1975comparison}.
As above in \eqref{eq:DefRj0} we consider the geodesic
$\gamma$ connecting $m=\gamma(0)$ and $u_{j_0}=\gamma(1).$
Along $\gamma$ we consider the Jacobi fields $J$ with $J(0)=0$ 
(which correspond to geodesic variations along $\gamma$ leaving $m$ fixed.) 
Let $(w^n)_n$ be an orthonormal basis of eigenvectors of the self-adjoint Jacobi operator 
$J \mapsto R(\frac{\gamma'(0)}{\|\gamma'(0) \|},J)\frac{\gamma'(0)}{\|\gamma'(0) \|},$ 
where $R$ denotes the Riemannian curvature tensor.
For convenience, we let $w^1$ be tangent to $\gamma.$
For each $n$, we denote the eigenvalue associated with $w^n$ by $\lambda_n.$ 
We use this basis to express the adjoint $R_{j_0}^\ast$ of $R_{j_0}$ of Theorem~\ref{thm:ComputeDerivativeOfM} given by \eqref{eq:DefRj0}.
Using the expression \eqref{eq:Littlerjay}, we get that 
\begin{equation}\label{eq:DiffOfLogInSym}
w= \sum\nolimits_n \alpha_n  w^n \ \mapsto \ R_{j_0}^\ast w =  A_{ij_{0}}\sum\nolimits_n \alpha_n f_1(\lambda_n) \ \pt_{m,u_{j_0}} w^n.
\end{equation}
Here, $\pt_{m,u_{j_0}} w^n$ denotes the parallel transport of the basis vector $w^n$ from the point $m$ to the point $u_{j_0}.$ 
The coefficients $\alpha_n$ are the coefficients of $w$ w.r.t.\ the othogonal basis  $w^n$ and 
the function $f_1,$ depending on the sign of $\lambda_n$ is given by 
\begin{equation}\label{eq:Fcalc4Rj}
f_1(\lambda_n) = 
\begin{cases}
1,   & \quad \text{if} \ \lambda_n = 0, \\
\frac{\sqrt{\lambda_n} d}{\sin (\sqrt{\lambda_n} d)},    & \quad \text{if} \ \lambda_n > 0, \quad  d < \pi/\sqrt{\lambda_n},\\
\frac{\sqrt{-\lambda_n} d}{\sinh (\sqrt{-\lambda_n} d)},  &  \quad \text{if} \ \lambda_n < 0,
\end{cases}
\end{equation}
where $d= \dist(m,u_{j_0}).$ For a derivation of \eqref{eq:Fcalc4Rj}, we refer to \cite{bredies2017total}.

Similarly, we may compute $l_{j_0}$ defined by \eqref{eq:littlelJay}.
We may use the same geodesic $\gamma$ we used above to connect $m=\gamma(0)$ and $u_{j_0}=\gamma(1),$
and the orthonormal basis $(w^n)_n$ of eigenvectors of the self-adjoint Jacobi operator $R$ w.r.t.\ the geodesic $\gamma$ above.
In contrast, we consider the Jacobi fields $J$ with $J(1)=0$
(which correspond to geodesic variations along $\gamma$ leaving $u_{j_0}$ fixed.)
We get, for $l_{j_0}^\ast: \mathcal{TM}_{m} \to \mathcal{TM}_{m},$
\begin{align}\label{eq:ComputeljotIn}
l_j^\ast :  \ w= \sum_n \alpha_n  w^n  \mapsto  l_j^\ast (w) = \sum_n f_2(\lambda_n) \ \alpha_n \  w^n, 
\end{align}
where 
the function $f_2,$ depending on the sign of $\lambda_n$ is given by 
\begin{equation}\label{eq:Fcalc4Rj2}
f_2(\lambda_n) = 
	 \begin{cases}
	     -1,   &\text{if} \quad \lambda_n = 0,  \\
	     - d \sqrt{\lambda_n} \ \frac{  \cos( \sqrt{\lambda_n} d)}{\sin( \sqrt{\lambda_n} d)},  
	     & \text{if} \quad \lambda_n > 0, \quad  d < \pi/\sqrt{\lambda_n}, \\
	     - d \sqrt{-\lambda_n} \ \frac{  \cosh( \sqrt{-\lambda_n} d)}{\sinh( \sqrt{-\lambda_n} d)}, 
	     & \text{if} \quad \lambda_n < 0.
	\end{cases}	     
\end{equation}
Using \eqref{eq:DefL}, we can, in a symmetric space, compute $L^\ast$ by 
\begin{align}
L^\ast w = \sum\nolimits_{j=1}^K A_{i{j}} \ l_j^\ast \ w, \quad \text{with} \quad  l^\ast_j \ w  \text{ given by } \eqref{eq:ComputeljotIn}.
\end{align}

\paragraph{Proximal Mappings for the Atoms of the Indirect Measurement Term.}

If we employ the cyclic proximal point scheme \eqref{eq:algCPPA}, we have to compute the proximal mappings 
of the atoms $\mathcal D_i$ of the data term, i.e., 
\begin{align}\label{eq:prox4Data}
 \prox_{\mu \mathcal D_i} u = 
 \argmin_{u' \in \mathcal{M^K}} \dist(\mathcal A(u')_i,f_i)^p + \frac{1}{2 \mu} \sum\nolimits_{k} \dist(u_k,u_k')^2, \qquad \mu > 0. 
\end{align}
This is particularly interesting, if the exponent $p$ of the distance function equals $one$ since, in this case,  
the distance function is not differentiable on the diagonal. 

Since for the case of an indirect measurement term no closed form expression of the proximal mapping \eqref{eq:prox4Data}
seems available,
we use a gradient descent scheme for $p>1,$  and a subgradient descent scheme for $p=1$, 
in order to compute the proximal mapping of the atoms $\mathcal D_i.$ 
Subgradient descent has already been used used to compute the proximal mappings of the $\TV_2$
and the $\TGV$ terms  in a manifold setting in \cite{bavcak2016second} and in \cite{bredies2017total}, respectively.
(Only for the $\TV$ term, a closed form proximal mapping was derived \cite{weinmann2014total}.)

Let us explain the (sub)gradient descent scheme to compute \eqref{eq:prox4Data}. 
We first observe that, by \eqref{eq:gradient12Dist}, the gradients 
of the mapping $u_k' \mapsto \dist(u_k,u_k')^2$ are given by  $-\log_{u'_{k}}u_{k}.$ 
The gradient of the first summand in \eqref{eq:prox4Data} was computed in the previous paragraph for $p>1$;  
see \eqref{eq:defDistFixSecond}, \eqref{eq:MapToDeriveGrad} and the following derivations of the adjoint of the differential of the mapping $m.$ If $p=1,$ we consider the (sub)gradient of $\dist_{f_i}(y) = \dist (y,f_i)$ which is, for $y \neq f_i$,
given by \eqref{eq:MapToDeriveGrad} as well and concatenate it with the the adjoint of the differential of the mapping $M;$
cf. \eqref{eq:thmEq2}.

\section{Experimental Results}\label{sec:Experiments}

\begin{figure}[!tp]
\def\figfolderA{experiments/compare_kernels_S1/}
\def\figfolderB{experiments/compare_kernel_Pos3/}
\def\hs{\hfill}
\def\vs{\vspace{0.03\textwidth}}
\def\figurewidth{0.3\textwidth}
\centering
{
\footnotesize
\begin{tabular}{ccc}
\includegraphics[width=\figurewidth]{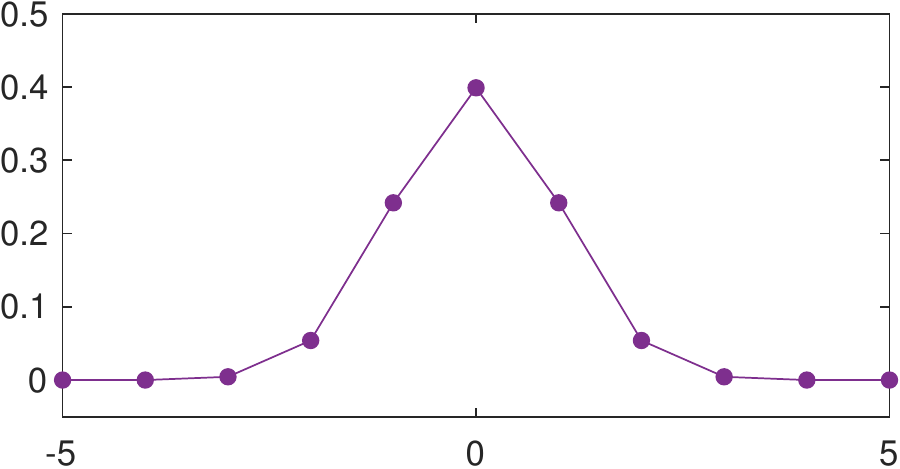} &
\includegraphics[width=\figurewidth]{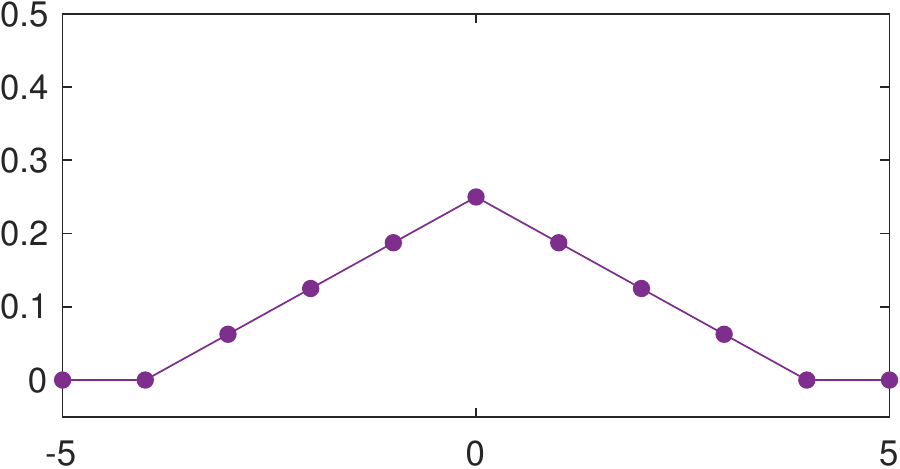} &
\includegraphics[width=\figurewidth]{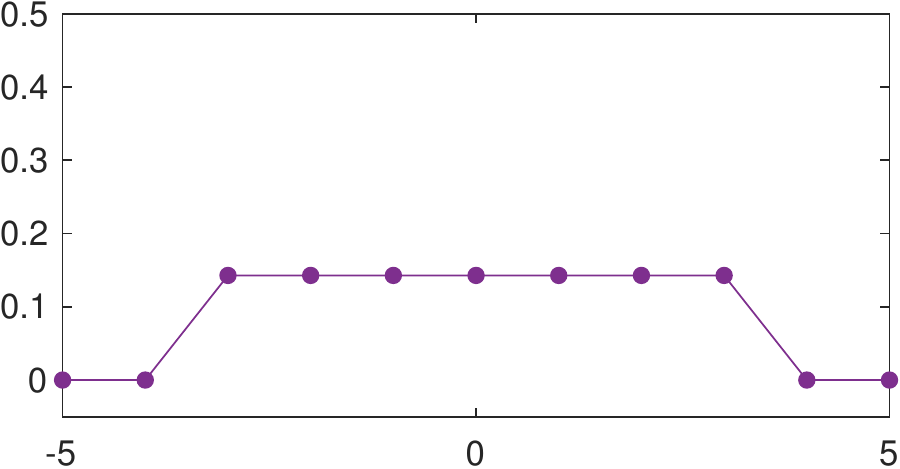} \\[3ex]
\includegraphics[width=\figurewidth]{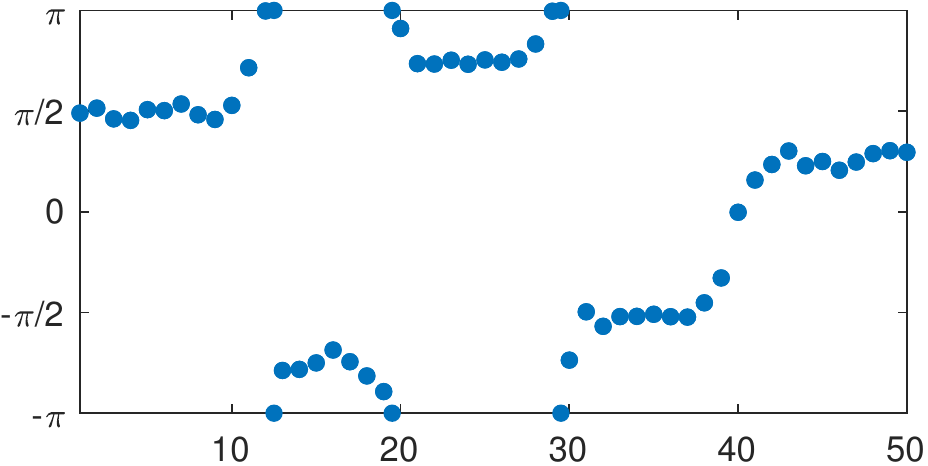}   &
\includegraphics[width=\figurewidth]{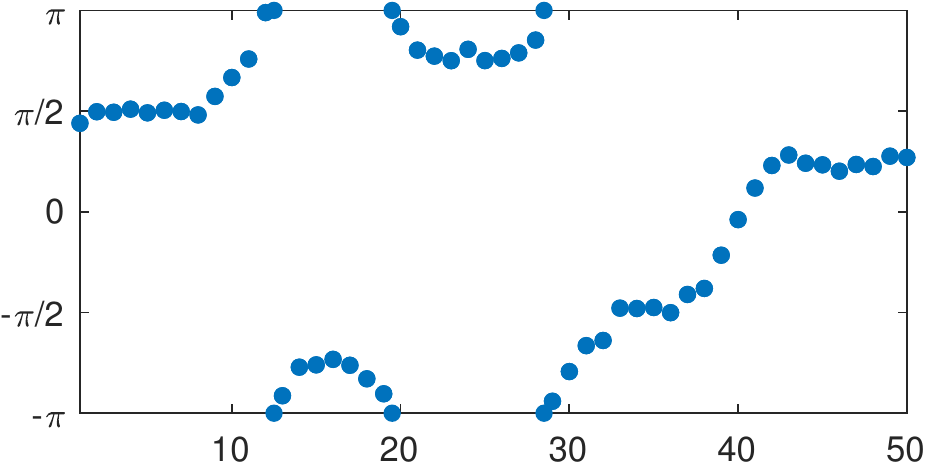}   &
\includegraphics[width=\figurewidth]{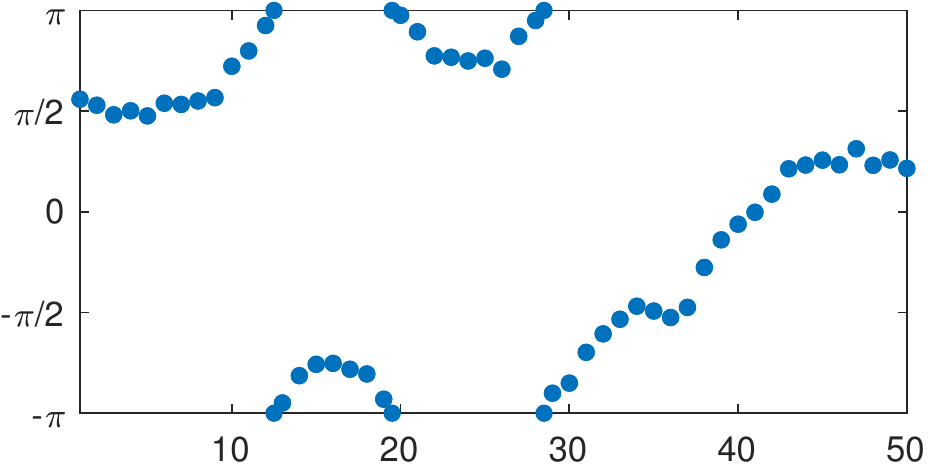}   \\[2ex]
\includegraphics[width=\figurewidth]{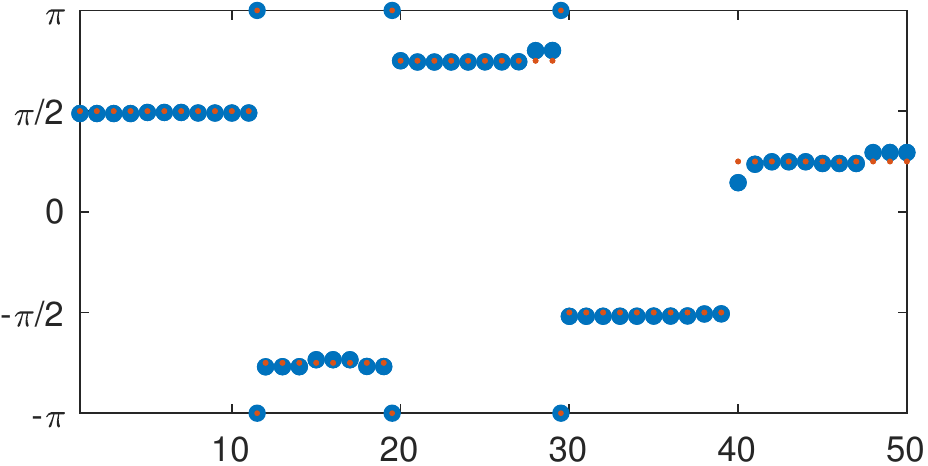} &
\includegraphics[width=\figurewidth]{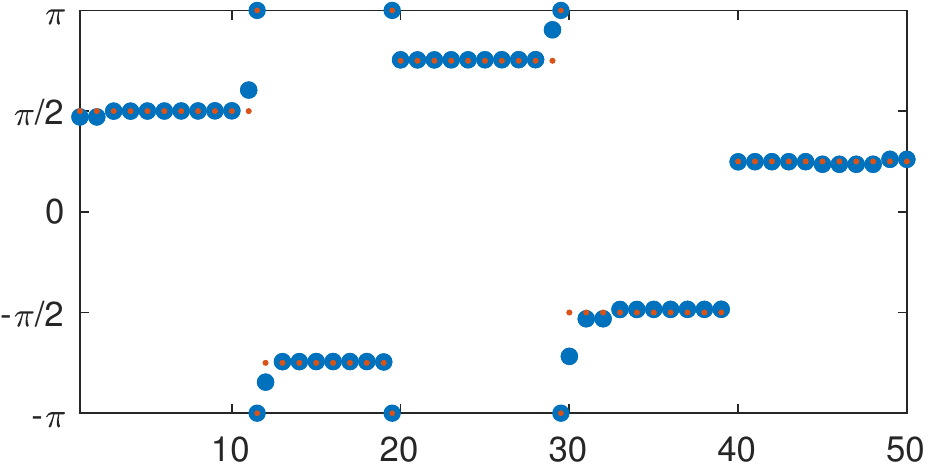} &
\includegraphics[width=\figurewidth]{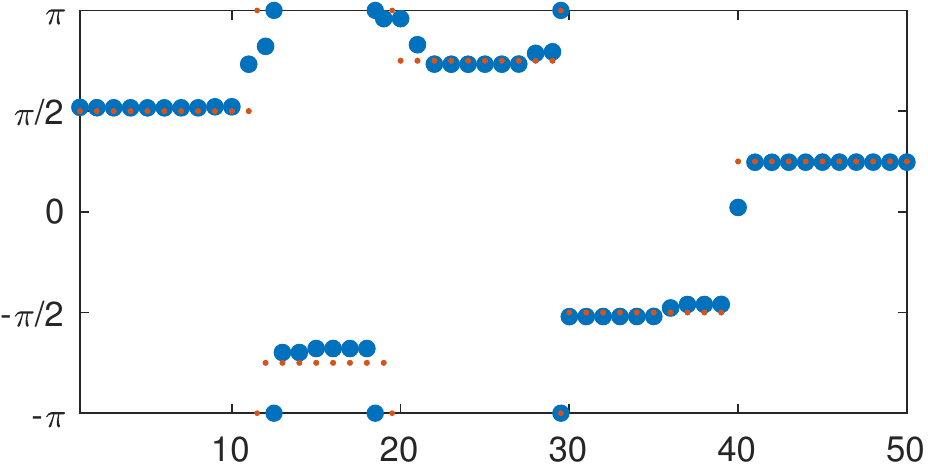} \\
$\deltaSNR$: $\input{\figfolderA exp_compare_kernels_deltaSNR_1.txt}$ dB &
$\deltaSNR$: $\input{\figfolderA exp_compare_kernels_deltaSNR_3.txt}$ dB&
$\deltaSNR$: $\input{\figfolderA exp_compare_kernels_deltaSNR_2.txt}$ dB \\[4ex]
\includegraphics[width=\figurewidth]{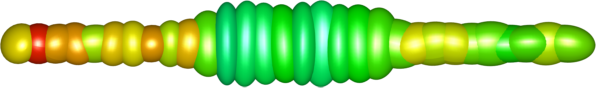}   &
\includegraphics[width=\figurewidth]{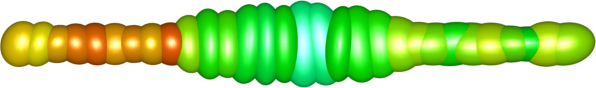}   &
\includegraphics[width=\figurewidth]{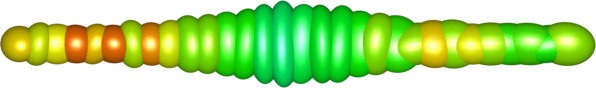}   \\[3ex]
\includegraphics[width=\figurewidth]{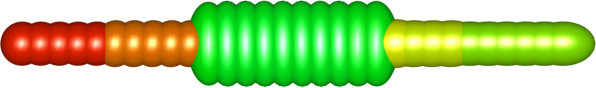} &
\includegraphics[width=\figurewidth]{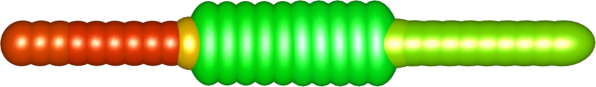} &
\includegraphics[width=\figurewidth]{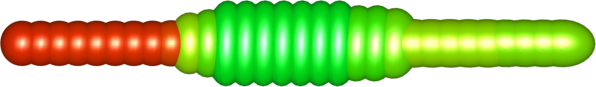} \\[2ex]
$\deltaSNR$: $\input{\figfolderB exp_compare_kernels_Pos3_deltaSNR_1.txt}$ dB &
$\deltaSNR$: $\input{\figfolderB exp_compare_kernels_Pos3_deltaSNR_3.txt}$ dB &
$\deltaSNR$: $\input{\figfolderB exp_compare_kernels_Pos3_deltaSNR_2.txt}$ dB
\end{tabular}
}
\caption{
Deconvolution results for 
a synthetic $S^1$-valued signal 
and a synthetic $\Pos_3$-valued signal
for 
a Gaussian kernel with $\sigma = 1$ \emph{(left)}, 
a triangular kernel \emph{(middle)}, and a moving average kernel \emph{(right)},
all of support size $7.$
\emph{From top to bottom:} Visualization of the kernel, convolved noisy $S^1$-valued signal, 
 deconvolution result (TV regularization with $p=2,$ $\lambda = 0.1$, groundtruth shown as red dashed line), convolved noisy $\Pos_3$-valued signal,  deconvolution result (TV regularization with $p=1,$ $\lambda = 1$).
}
\label{fig:s1_signal}
\end{figure}

\paragraph{Experimental setup.}
We carry out experiments for data with values in the circle $S^1$, the sphere $S^2$ and 
 the manifold $\Pos_3$ of positive definite matrices equipped with the Fisher-Rao metric.
 $S^1$ valued data is visualized by the phase angle, and color-coded as hue value in the HSV color space when displaying image data.
 We visualize $S^2$ valued data
by a color coding based on Euler angles as shown in Figure~\ref{fig:S2_img}.
Data on the $\Pos_3$ manifold is visualized
by the isosurfaces of the corresponding quadratic forms. More precisely, the ellipse visualizing
the point $f_p$ at voxel $p$ are the points $x$ fulfilling $(x-p)^\top f^{-1}_p (x-p) = c,$ for some $c>0.$
To quantitatively measure the  quality of a reconstruction,
we use the manifold variant of the \emph{signal-to-noise ratio improvement} 
$
	\deltaSNR = 10 \log_{10} \left(  \sum_{ij} d(g_{ij}, f_{ij})^2  / \sum_{ij} d(g_{ij}, u_{ij})^2\right) \dB,
$
see \cite{weinmann2014total}.
Here $f$ is the noisy data, $g$ is the ground truth, and $u$ is a regularized  reconstruction.
A higher $\deltaSNR$ value  means better reconstruction quality.
As in \cite{bredies2017total}, 
we parametrize the model parameters  $\lambda_0,\lambda_1$ of TGV
by 
$\lambda_0 =  r \frac{(1 - s)}{s'},$ and by $\lambda_1 =  r \frac{s}{s'},$
with $s' = \min(s, 1-s)$
so that  $r \in (0, \infty)$ controls the overall regularization strength 
and $s \in (0,1)$ the balance of the two TGV penalties.
All examples were computed using 1000 iterations.
We have implemented the presented methods in Matlab 2017b.
We used functions of the toolboxes CircStat~\cite{berens2009circstat}, Manopt~\cite{manopt}, MVIRT~\cite{bavcak2016second}, 
and implementations from the authors' prior works \cite{weinmann2014total, bredies2017total}.

\begin{figure}
\def\figfolder{experiments/compare_algos_S1_S2_Pos3/}
\def\hs{\hfill}
\def\vs{\vspace{0.03\textwidth}}
\def\figurewidth{0.8\textwidth}
\centering
\includegraphics[width=\figurewidth]{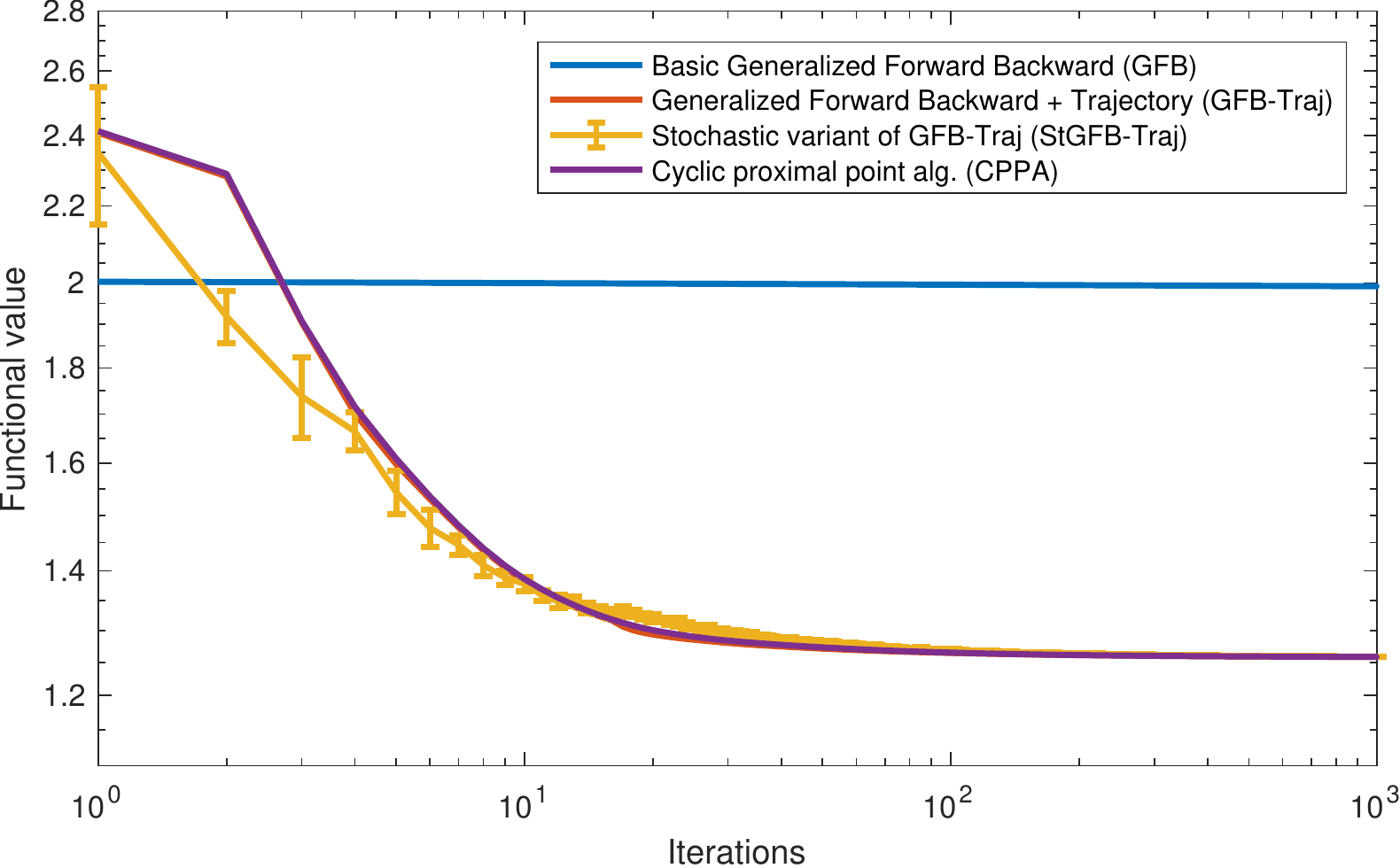} 
\caption{
Functional value of the different algorithmic variants 
in dependence on the number of iterations for the setup in Figure~\ref{fig:s1_signal} ($S^1$ valued data with Gaussian kernel).
The basic generalized forward-backward scheme leads to very small decrements of the functional value. 
The cyclic proximal point algorithm, the trajectory variant 
and their stochastic variants (displayed averages over ten runs and error bars) provide a much faster degression.
}
\label{fig:compare_algos_S1}
\end{figure}

\paragraph{Results for basic univariate signals.}
We start with reconstruction results of univariate signals
where we for the moment focus on the prototypical case of TV regularization, i.e., $R(u) = \TV(u).$ 
In the first experiment, we performed manifold-valued convolutions on piecewise constant signals 
with three different types of kernels:
Gaussian kernels, triangular kernels, and moving average kernels; see Figure~\ref{fig:s1_signal}.
The $S^1$-valued signal was corrupted by von Mises noise with concentration parameter $\kappa = 100.$
Here, we use the parameter $p=2$ in the data fidelity term.
The $\Pos_3$ valued signal was corrupted by Rician noise of level $30.$ 
We use $p=1$ for the $\Pos_3$ valued data as it is more suitable for this kind of data; cf. \cite{weinmann2014total}. The shown reconstructions were computed using 
 the deterministic trajectory variant for the $S^1$-valued signal, and by the cyclic proximal point variant for the $\Pos_3$ valued signal.
Qualitatively, the resulting signals are close to the ground truth.
In particular, the jumps, which were smoothed out by the convolution, are recovered.
Note that the phase jump between $-\pi$ and $\pi$
is properly taken into account for the spherical data.
Quantitatively, the results exhibit significant signal-to-noise ratio improvements
ranging from  $3.6$~dB up to  $11.6$~dB. 

\paragraph{Comparison of the proposed algorithmic variants.}
We compare the performance of the 
four algorithmic variants discussed above:
the basic generalized forward backward scheme (GFB), 
the generalized forward backward scheme with trajectory method (GFB-Traj),
its sto\-chastic variant (StGFB-Traj), and
the cyclic proximal point algorithm (CPPA).
To this end, we investigate the evolutions of the functional 
values over the iterations using the $S^1$-data of Figure~\ref{fig:s1_signal}.
The results are reported in Figure~\ref{fig:compare_algos_S1}. 
The basic GFB exhibits slow decay of the functional value
whereas the other three variants achieve a much better  decay.
The reason for this is that the basic GFB typically demands for globally short step sizes
(see the discussion in Section~\ref{sec:Basic_Algorithmic_Structures})
whereas the other schemes adapt to the local situation.
The graphs suggest that
GFB-Traj, StGFB-Traj and CPPA perform similarly good. 
Since in our reference implementation the trajectory variants
were more than ten times faster than the CPPA variant,
we recommend to use the trajectory methods for the case $p=2.$
(However, for $p=1,$ we use the CPPA which in this case is the only applicable scheme.)

\paragraph{Reconstruction results for manifold valued images.}

\begin{figure}

\def\figfolderA{experiments/exp_S1_img_Gauss/}
\def\figfolderB{experiments/exp_S1_img/}
\def\hs{\hspace{0.03\textwidth}}
\def\vs{\vspace{0.03\textwidth}}
\def\figurewidth{0.22\textwidth}
\centering
\includegraphics[width=\figurewidth]{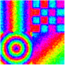} \hs
\includegraphics[width=\figurewidth]{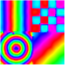} \hs\hs
\includegraphics[width=\figurewidth]{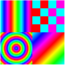} \\[2ex]
\includegraphics[width=\figurewidth]{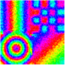} \hs
\includegraphics[width=\figurewidth]{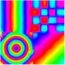} \hs\hs
\phantom{\includegraphics[width=\figurewidth]{\figfolderB exp_S1_img_result}}
\caption{
Deconvolution results for an $S^1$-valued image.
\emph{From left to right:} input data, reconstruction result using TGV regularized deconvolution ($p=2,$ $r = 0.2,$
$s = 0.3)$, and groundtruth, all visualized as hue values.
\emph{Top:}
ground truth convolved with a Gaussian kernel ($5\times 5$ kernel with $\sigma = 1$) and corrupted by von Mises noise;
$\deltaSNR$: $\protect\input{\figfolderA/exp_S1_deltaSNR.txt}$~dB.
\emph{Bottom:} ground truth convolved with a $5\times 5$ moving average kernel and corrupted by von Mises noise; $\deltaSNR$: $\protect\input{\figfolderB/exp_S1_deltaSNR.txt}$~dB.
}
\label{fig:S1_img}
\end{figure}

\begin{figure}
\def\figfolder{experiments/exp_S2_img/}
\def\hs{\hfill}
\def\vs{\vspace{0.03\textwidth}}
\def\figurewidth{0.22\textwidth}
\centering
\includegraphics[width=\figurewidth]{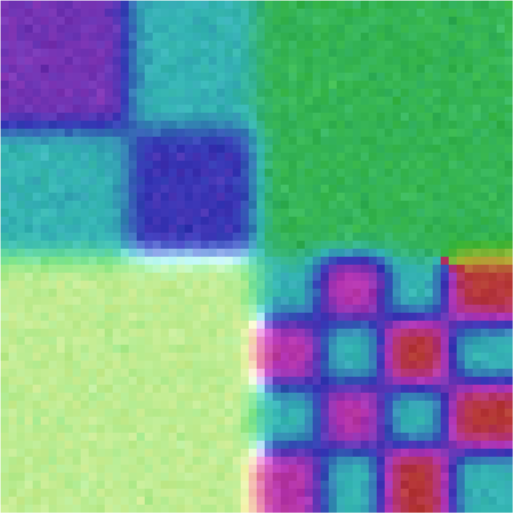} \hs
\includegraphics[width=\figurewidth]{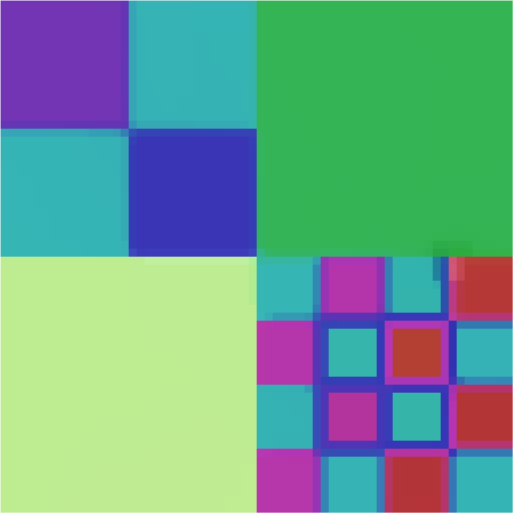} \hs\hs\hs
\includegraphics[width=\figurewidth]{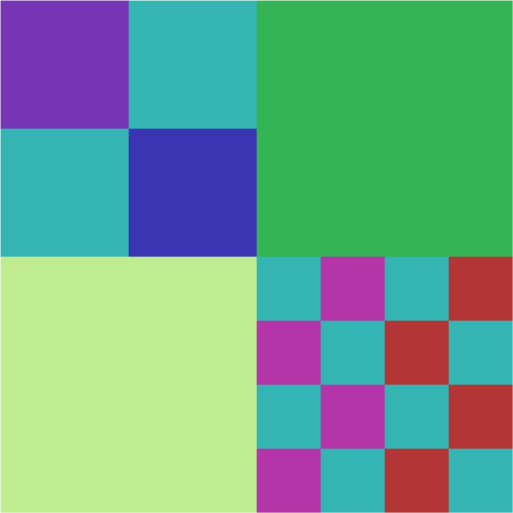} \hs
\includegraphics[width=\figurewidth]{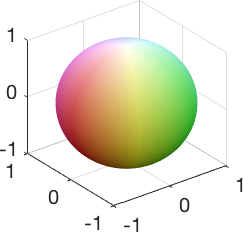}
\caption{
Deconvolution results for an $S^2$-valued image.
\emph{From left to right:} image convolved with a
 $5\times 5$ Gaussian kernel with $\sigma = 1.5$ and corrupted by wrapped Gaussian noise;
TV regularized deconvolution ($p=2,$ $\lambda = 0.1$, $\deltaSNR$: $\protect\input{\figfolder/exp_S2_img_deltaSNR.txt}$ dB); ground truth; 
 color code for the visualization of $S^2$-valued images. 
}
\label{fig:S2_img}
\end{figure}

\begin{figure}
\def\figfolder{experiments/exp_SAR/}
\def\hs{\hspace{0.03\textwidth}}
\def\vs{\vspace{0.03\textwidth}}
\def\figurewidth{0.3\textwidth}
\centering
\includegraphics[width=\figurewidth]{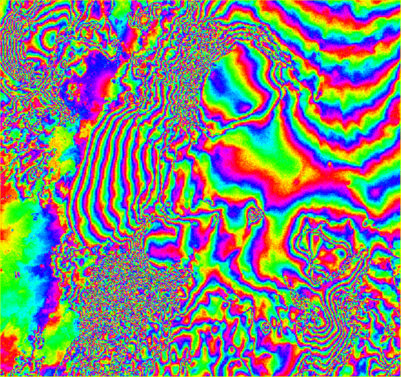} \hs
\includegraphics[width=\figurewidth]{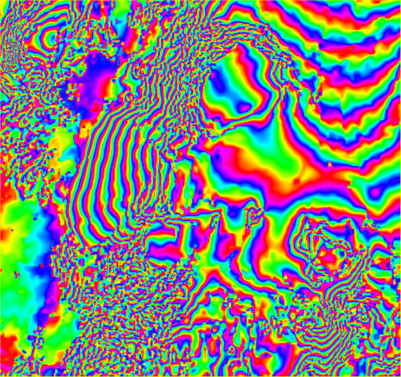} 
\caption{
\emph{Left:} Interferometric SAR image from \cite{thiel1997ers} ($S^1$-valued data). 
	\emph{Right:} Result of TGV regularized deconvolution ($p=2,$ $r = 0.2,$ $s = 0.3$)
	using a $5\times 5$ Gaussian kernel with $\sigma = 1.$ 
}
\label{fig:sar}
\end{figure}

\begin{figure}
\def\figfolderA{experiments/exp_Pos3_img_Gauss/}
\def\figfolderB{experiments/exp_Pos3_img/}
\def\hs{\hspace{0.03\textwidth}}
\def\vs{\vspace{0.03\textwidth}}
\def\figurewidth{0.22\textwidth}
\centering
\includegraphics[width=\figurewidth]{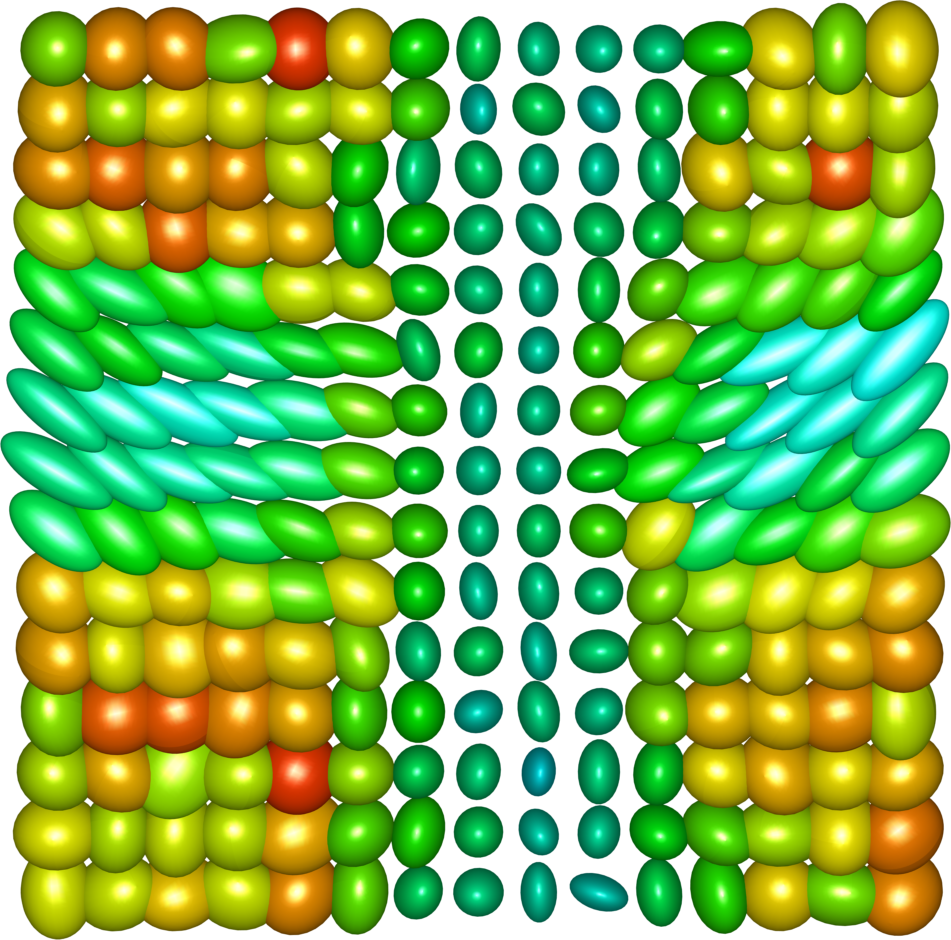} \hs
\includegraphics[width=\figurewidth]{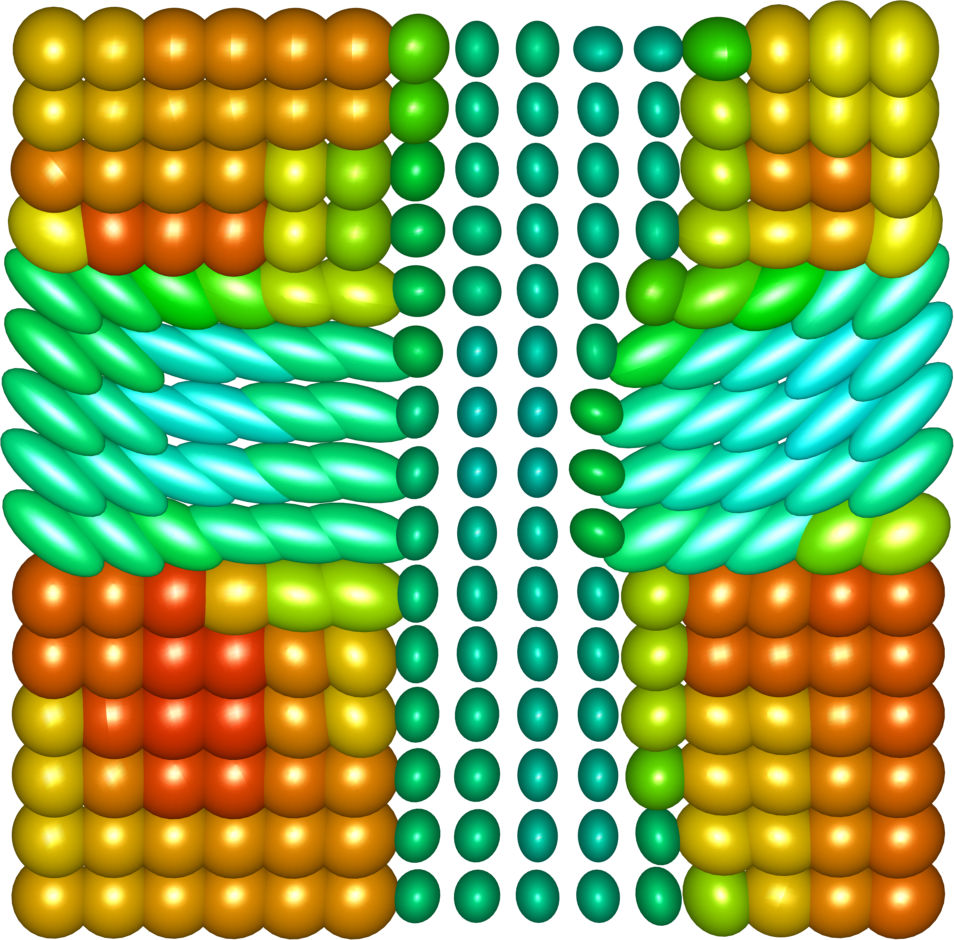} \hs\hs
\includegraphics[width=\figurewidth]{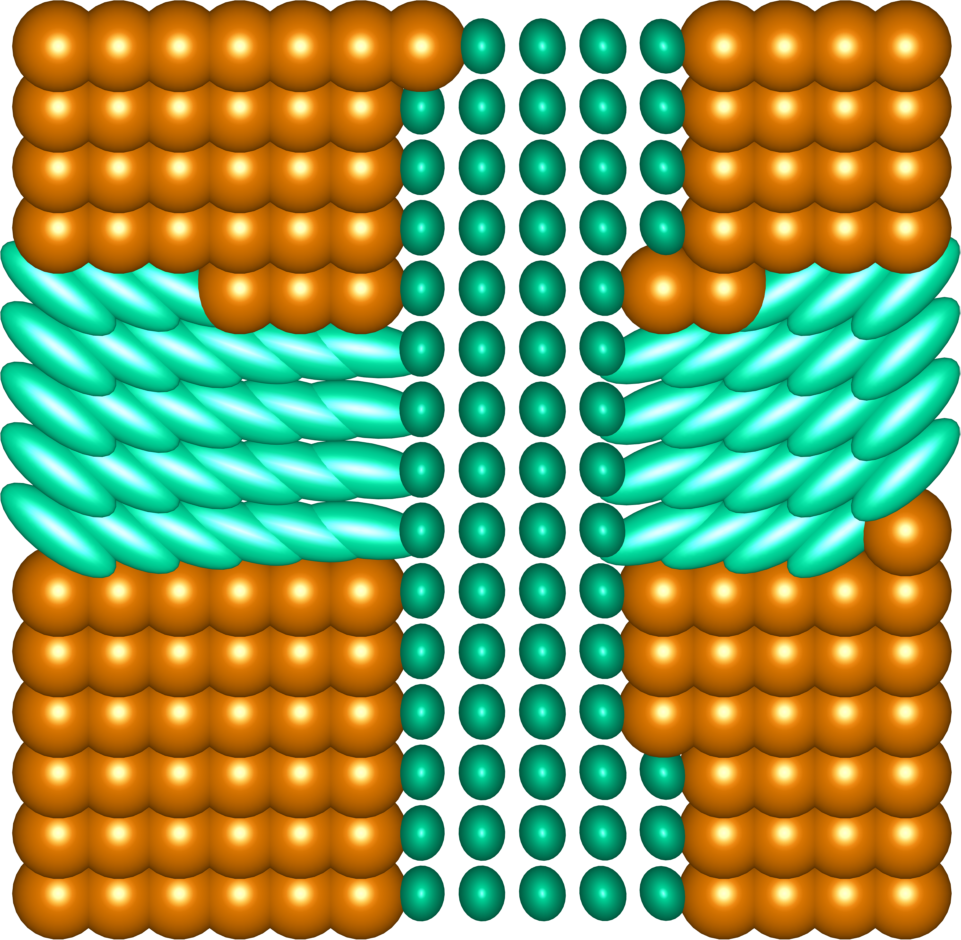} \\[2ex]
\includegraphics[width=\figurewidth]{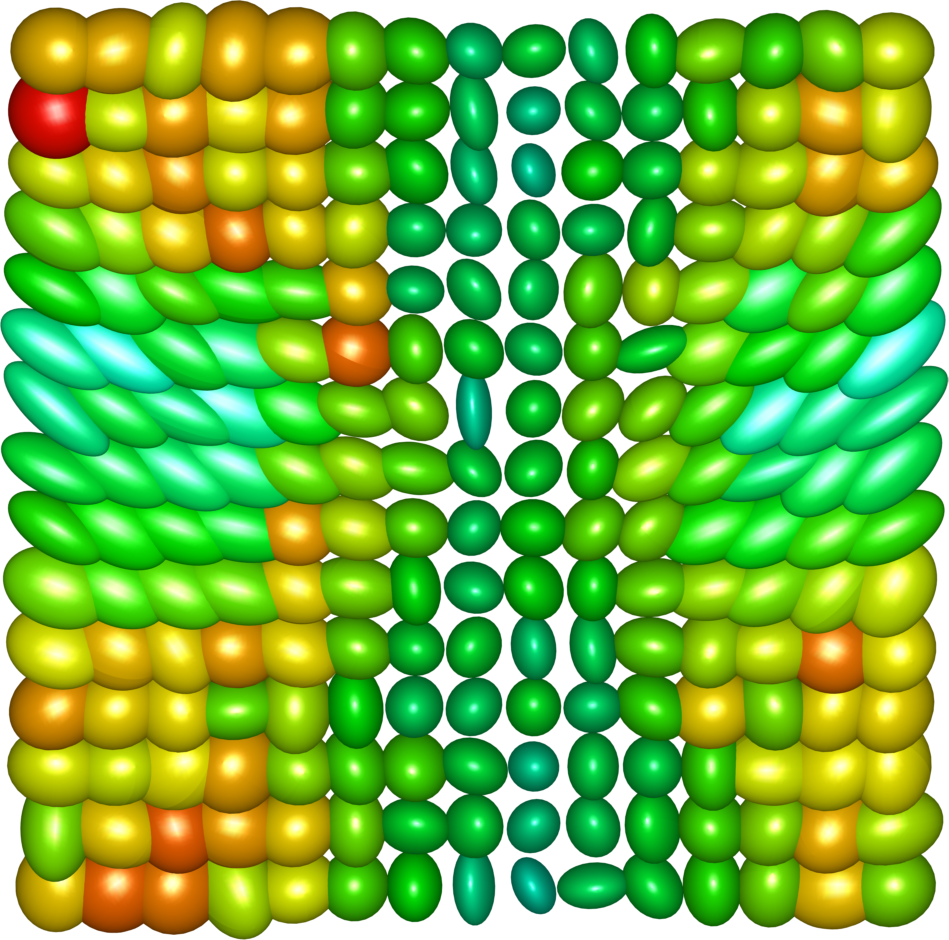} \hs
\includegraphics[width=\figurewidth]{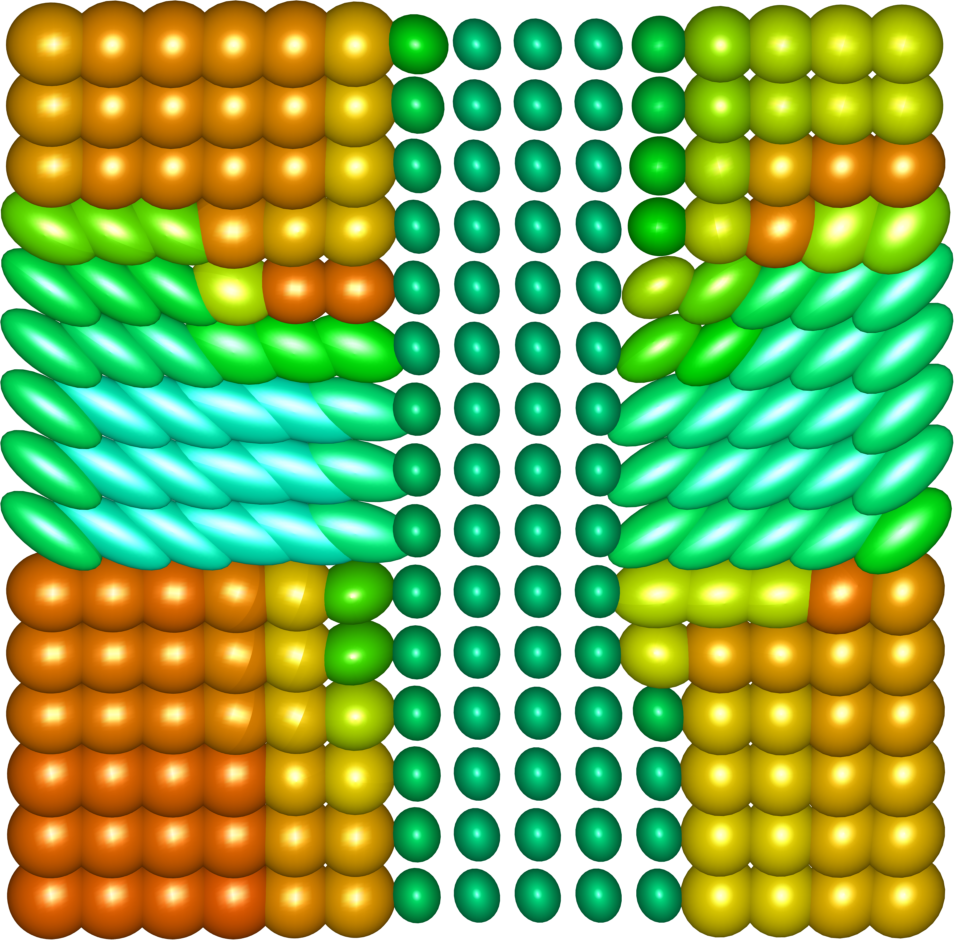} \hs\hs
\phantom{\includegraphics[width=\figurewidth]{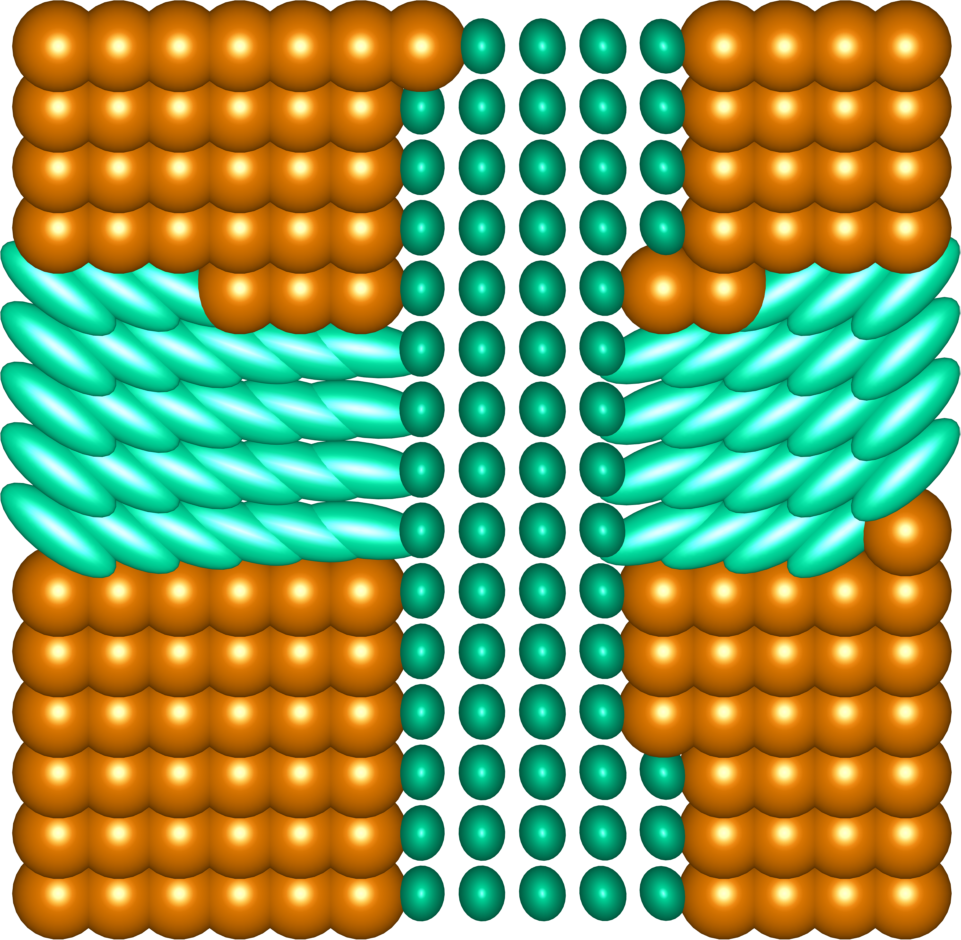}}
\caption{
Deconvolution results for a synthetic $\Pos_3$ valued image.
\emph{From left to right:} input data, reconstruction result using 
TV regularization  ($p=1,$ $\lambda = 0.1),$ and groundtruth.
\emph{Top:}
ground truth convolved with a Gaussian kernel ($5\times 5$ kernel with $\sigma = 1$) and corrupted by Rician noise;
$\deltaSNR$: $\protect\input{\figfolderA/exp_Pos3_deltaSNR.txt}$ dB. 
\emph{Bottom:} groundtruth convolved with a $5\times 5$ moving average kernel and corrupted by Rician noise; 
$\deltaSNR$: $\protect\input{\figfolderB/exp_Pos3_deltaSNR.txt}$ dB.}
\label{fig:Pos3_img}
\end{figure}

\begin{figure}
\def\figfolder{experiments/exp_camino/}
\def\hs{\hspace{0.03\textwidth}}
\def\vs{\vspace{0.03\textwidth}}
\def\figurewidth{0.45\textwidth}
\centering
\includegraphics[width=\figurewidth]{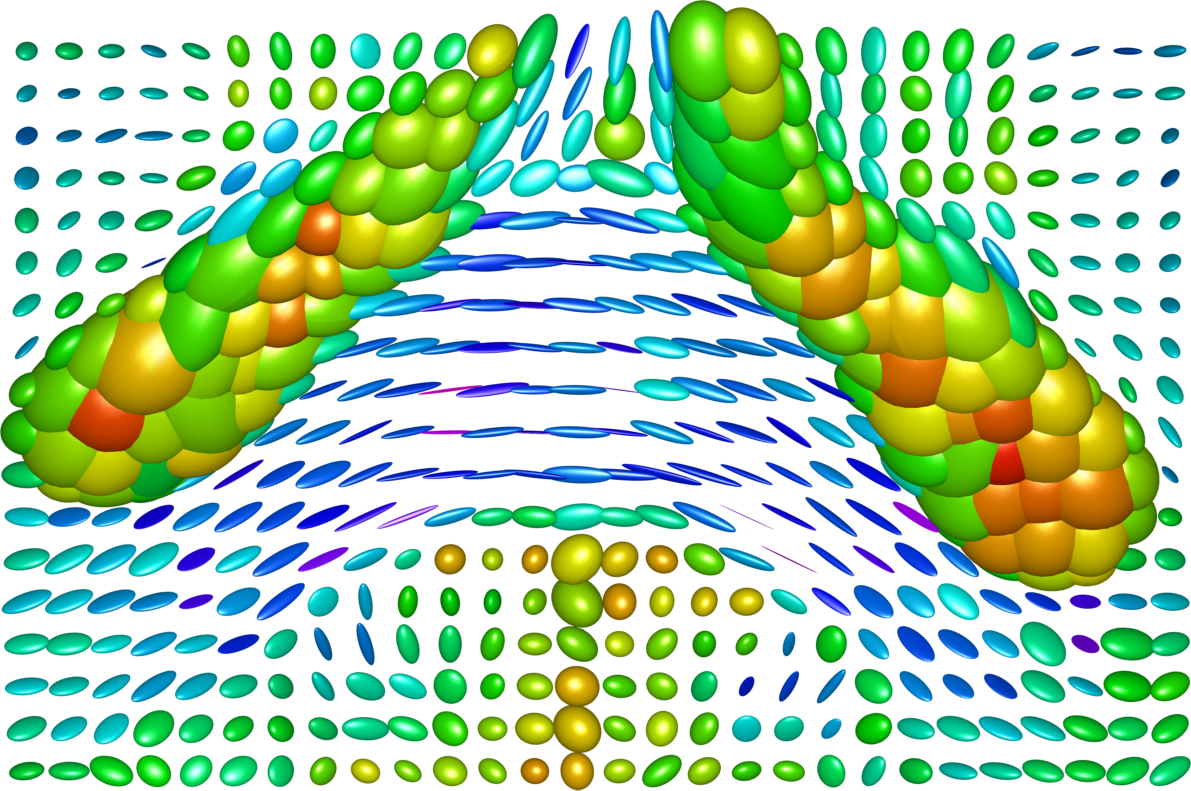} \hs
\includegraphics[width=\figurewidth]{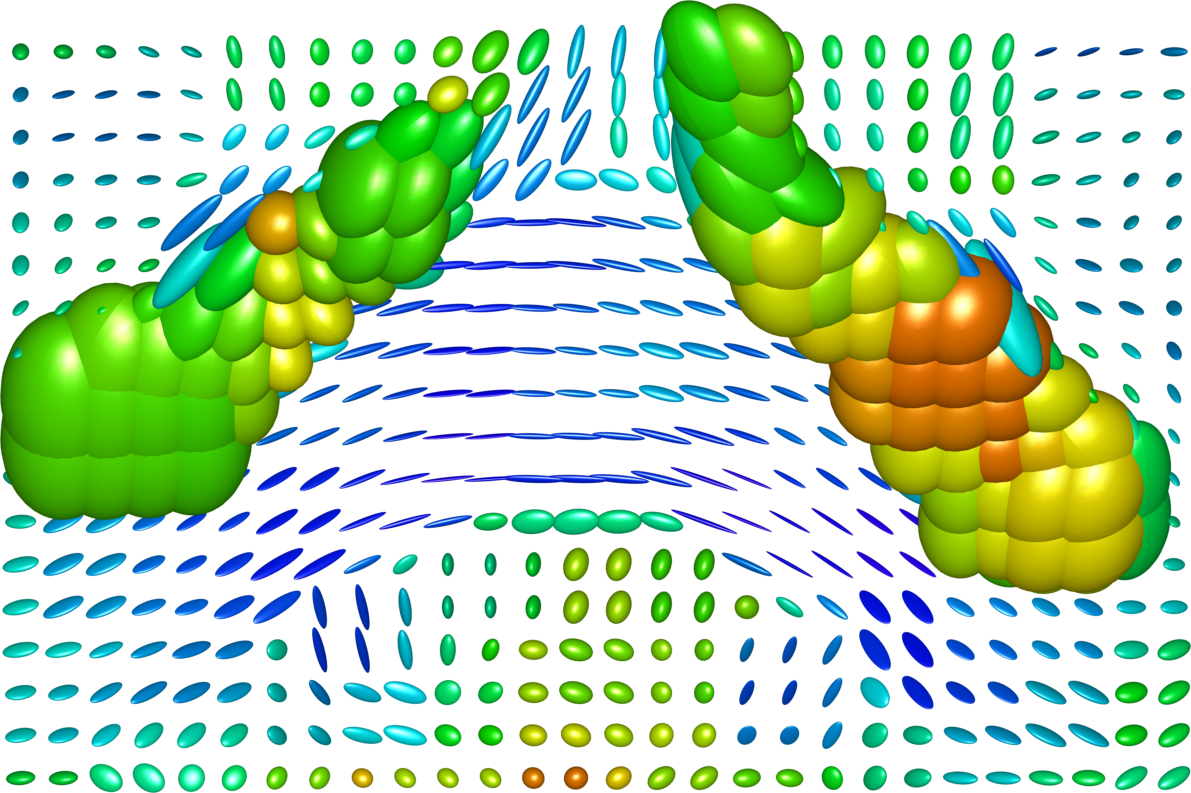} 
\caption{
\emph{Left:}  Diffusion tensor image of a human brain (axial slice)
showing the corpus callosum ($\Pos_3$-valued).
\emph{Right:} TV regularized deconvolution ($p=1,$ $\lambda = 0.1$)
using a $5 \times 5$ Gaussian kernel with $\sigma = 1.$
}
\label{fig:camino}
\end{figure}

We illustrate the effect of the proposed methods for  manifold-valued images.
As before, we use $p=1$ for the $\Pos_3$-valued images, and $p=2$ otherwise.
Following the above discussion,
we use the  CPPA method for $p=1$
and the trajectory method for $p=2.$  
First we consider spherical data.
Figure~\ref{fig:S1_img} shows the TGV-regularized deconvolution of a synthetic  $S^1$-valued image,
and Figure~\ref{fig:S2_img} the TV-regularized deconvolution of a synthetic $S^2$-valued image.
We observe  that the proposed reconstruction method is effective:
the blurred edges are sharpened 
and the signal-to-noise ratio is significantly improved.
We illustrate the effect of the proposed method on a real interferometric synthetic aperture radar (InSAR) image.
Synthetic aperture radar (SAR) is a radar technique for
sensing the earth's surface from measurements taken by aircrafts or satellites. 
InSAR images consist of the phase difference between two SAR images,
recording a region of interest either from two different angles of view or at two different points in time.
Important  applications of InSAR are the creation of accurate digital elevation models and the detection of terrain changes; cf.~\cite{massonnet1998radar,rocca1997overview}.
As InSAR data consists of phase values, their natural data space is the unit circle.
Figure~\ref{fig:sar} shows the effect of the proposed method on a real InSAR image taken from \cite{thiel1997ers}\footnote{Data available at \url{https://earth.esa.int/workshops/ers97/papers/thiel/index-2.html}.}. 
Next we consider $\Pos_3$-valued images.
Figure~\ref{fig:Pos3_img} shows the TV-regularized deconvolution of a synthetic image.
As in the spherical case, the edges are sharpened
and the signal-to-noise ratio is significantly improved.
A real data example is a diffusion tensor image of a human brain
provided by the Camino project \cite{cook2006camino}\footnote{Data available at \url{http://camino.cs.ucl.ac.uk/}}.
The original tensors were computed from the diffusion weighted images by a least squares fit
based on the Stejskal-Tanner equation, and invalid tensors were filled by averages of their neighboring pixels.
The image shows the corpus callosum which connects the right and the left hemisphere.
Figure~\ref{fig:camino} shows the effect of the proposed method
using a Gaussian kernel with $\sigma = 1.$

\section{Conclusion}\label{sec:Conclusion}

In this paper we have proposed and studied models for the variational (Tichanov-Phillips) regularization
with indirect measurement terms in the manifold setup.
In particular, the models apply to deconvolution/deblurring and to $\TV$ and $\TGV$ regularization for manifold-valued data in a multivariate setting. We have obtained results on the existence of minimizers of these models.
Further, we have provided the details for algorithmic realizations of the proposed variational
models. 
For differentiable data terms, we have further developed the concept of a generalized forward backward-scheme:
we have proposed a trajectory method together with a Gau\ss-Seidel type update scheme which improves the computational performance of the generalized forward backward-scheme
as confirmed in the experimental section.
We have also considered a variant based on a stochastic gradient descend part.    
For a non-differentiable data term, we have employed the well-established concept of a cyclic proximal point algorithm. 
For the implementation of these schemes we have 
derived rather explicit differential geometric representations of the (sub)gradients of the 
data terms which was a challenging task and a central contribution of this work.
In particular, we have computed explicit representations of the derivatives of the intrinsic mean mapping \eqref{eq:IntrMeanIntro}
w.r.t.\ the input points in the manifold.
Finally, we have provided a numerical study of the proposed schemes.
We have seen experiments with real and synthetic data. As data spaces we have considered the unit circle, the two-dimensional sphere as well as the space of positive matrices.

\section*{Acknowledgment}

Martin Storath was supported by the German Research Foundation DFG under Grant STO1126/2-1.
Andreas Weinmann was supported by the German Research Foundation DFG under Grants WE5886/4-1, WE5886/3-1.

\appendix

\section{Appendix}\label{sec:Appendix} 

We here supply the proofs for those statements of Section~\ref{sec:well-posedness} we have not shown right in place.

\begin{proof}[Proof of Lemma \ref{lem:TV2islsc}]
	In order to show that the $\TV_2$ regularizer is lower semicontinuous, we show that 
	the atom $d_{1,1}$ of \eqref{eq:DefTV2} is lower semicontinuous. The lower semicontinuity of  
	$d_{1,1}$ implies the lower semicontinuity since $d_{2}(x,y,z) = d_{1,1}(x,y,y,z)$
	which, having a look at the decomposition in \eqref{eq:DefTV2}, in turn, 
	implies the lower semicontinuity of $\TV_2.$
	
	To show the lower semicontinuity of $d_{1,1},$ we let $x_n \to x,$
	$y_n \to y,$ $v_n \to v,$ and $z_n \to z$ in $\mathcal M$ as $n \to \infty.$
	We have to show that 
	$d_{1,1}(x,y,v,z) \leq \liminf_{n \to \infty}  d_{1,1}(x_n,y_n,v_n,z_n). $
	To this end, we let $x_{n_k},y_{n_k},v_{n_k},z_{n_k}$ be corresponding  subsequences such that 
	$\lim_{k \to \infty} d_{1,1}(x_{n_k},y_{n_k},v_{n_k},z_{n_k})=$ $\liminf_{n \to \infty}  d_{1,1}(x_n,y_n,v_n,z_n). $
	We let $a_k$ and $b_k$ be corresponding midpoints of $x_{n_k},z_{n_k}$ and $y_{n_k},v_{n_k},$ respectively, such that 
	$d_{1,1}(x_{n_k},y_{n_k},v_{n_k},z_{n_k})$ $= \dist (\midp(x_{n_k},z_{n_k}),\midp(y_{n_k},v_{n_k})) =$
	$\dist (a_k,b_k).$ As convergent sequences, each of the sequences $x_{n_k},y_{n_k},v_{n_k},z_{n_k}$ is bounded. 
	By Lemma~\ref{lem:MeansInBoundedSet}, 
	the sequences $a_k,b_k$ are bounded, and therefore by the completeness of the manifold $\mathcal M,$ there is a common choice of subindices $l \mapsto k_l$
	such that $a_{k_l}$ and $b_{k_l}$ both converge.
	Let $a_{k_l} \to a$ and $b_{k_l} \to b$ as $k \to \infty.$ 
	Then $a$ is a midpoint of $x,y$ and $b$ is a midpoint of $v,z,$ and so 
	\begin{align*}
	d_{1,1}(x,y,v,z) &\leq \dist (a,b) \leq \lim_l \dist (a_{k_l},b_{k_l}) \\
	&= \lim_{k \to \infty} d_{1,1}(x_{n_k},y_{n_k},v_{n_k},z_{n_k})     =\liminf_{n \to \infty}  d_{1,1}(x_n,y_n,v_n,z_n). 
	\end{align*}
	This shows that $d_{1,1}$ is lower semicontinuous and completes the proof. 		
\end{proof}

\begin{proof}[Proof of Theorem \ref{thm:ExistenceCondR4SecondOrd}]	
	We proceed as in the proof of Theorem~\ref{thm:ExistenceCondR} and note that since 
	the lower semicontinuity of the data term \eqref{eq:DefDist4Vectors} is shown in Lemma~\ref{lem:DatTermLSC},
	and the	lower semicontinuity of the regularizing term $R$ is assumed, it is enough to show that the functional $F$ of 
	\eqref{eq:FunctionalFinExProof} is coercive.  
	Towards a contradiction suppose that $F$ is not coercive. Then there is $\sigma \in \mathcal M^K$ and a sequence $u^{(n)}$ in
	$\mathcal M^K,$ such that   
	$\dist(u^{(n)},\sigma) \to \infty$ as well as a subsequence of $u^{(n)}$ with bounded value of $F,$ i.e., 
	there is a subsequence $u^{(n_k)}$ of $u^{(n)},$ as well as a constant $C''>0$ such that 
	$F(u^{(n_k)}) \leq C''$ for all $k \in \mathbb N.$
	Hence, $R(u^{(n_k)}) \leq C''.$   
	By our assumption on $R,$ this implies that $\diam(u^{(n_k)})$ does not converge to $\infty$ or that 
	$\dist(u^{(n_k)}_{l_s},u^{(n_k)}_{r_s})$ is not bounded in $k \in \mathbb N$ for some $s \in \{0,S\}.$
	The latter situation cannot happen by our assumption on $\mathcal A$ for the following reason:
	if $\dist(u^{(n_k)}_{l_s},u^{(n_k)}_{r_s})$ where unbounded, 
	$\diam (\mathcal A u^{(n_k)})$ were unbounded or $R(u^{(n_k)})$ were unbounded;
	$R(u^{(n_k)})$ is bounded by our assumption on $u^{(n)},$
	and $\diam (\mathcal A u^{(n_k)})$ is bounded since 
	$\dist(\mathcal A(u^{(n_k)}),f)^p$ is bounded in $n$
	since $F(u^{(n)})<C''$ for all $n \in \mathbb N$ by our assumption.
	In consequence,  $\dist(u^{(n_k)}_{l_s},u^{(n_k)}_{r_s})$ is bounded 
	by our assumption on $\mathcal A.$
	This implies that $\diam(u^{(n_k)})$ does not converge to $\infty.$
	At this point we can now follow the argument of the proof of Theorem~\ref{thm:ExistenceCondR}
	starting at \eqref{eq:diamBounded4allK} literally to conclude the assertion. 
\end{proof}

\begin{proof}[Proof of Theorem~\ref{thm:ExistenceMinimizersSecondOrder}]
	We apply Theorem~\ref{thm:ExistenceCondR4SecondOrd}. 
	The lower semicontinuity of the $\TV_2$ term is shown in Lemma~\ref{lem:TV2islsc}. 	
	Towards the other condition of Theorem~\ref{thm:ExistenceCondR4SecondOrd}, 
	let $u^{(n)}$ be  a sequence such that  $\diam(u^{(n)}) \to \infty,$ 
	and such that we have $\dist(u^{(n)}_{x_0,y_0},u^{(n)}_{x_0+1,y_0}) \leq C' $ as well as   
	$\dist(u^{(n)}_{x_1,y_1},u^{(n)}_{x_1,y_1+1}) \leq C' $ for some $C'>0$ and all $n \in \mathbb N.$ 
	We have to show that  $\TV_2(u^{(n)})\to \infty.$	
	Towards a contradiction, assume that there is a subsequence $u^{(n_k)}$ of $u^{(n)}$ and $C''>0$ such that $\TV_2(u^{(n_k)}) \leq C''.$ 
	We show that, since $\TV_2(u^{(n_k)}) \leq C'',$ there is a 
	constant $C'''>0$ such that $\dist(u^{(n_k)}_{x+1,y},u^{(n_k)}_{x,y}) \leq C''' $ and 
	such that $\dist(u^{(n_k)}_{x,y+1},u^{(n_k)}_{x,y}) \leq C'''$	for all $x,y.$
	This can be seen by iterative application of the following fact: if $\TV_2(v^{(n)}) \leq C'',$
	then $d_2,d_{1,1}$ are bounded by $C''$ as well, and 
	$\dist(v^{(n)}_{x_1,y+1},v^{(n)}_{x_1,y}) 
	\leq \dist(v^{(n)}_{x_1,y},\midp(v^{(n)}_{x_1,y-1},v^{(n)}_{x_1,y+1})) + C'' 
	\leq \dist(v^{(n)}_{x_1,y},v^{(n)}_{x_1,y-1}) + 2 C''.
	$ 
	Iterative application shows the validity for the ``cross'' of indices $(x,y)$ with either $x=x_0$ or $y=y_0$  
	with the respective direction along the ``cross''. 
	Considering the index $(x_0,y_1)$ and using the triangle inequality implies   
	$\dist(v^{(n)}_{x_0,y_1},v^{(n)}_{x_0+1,y_1+1})$    
	$= 2 \dist(v^{(n)}_{x_0,y_1}, \midp( v^{(n)}_{x_0+1,y_1+1},v^{(n)}_{x_0,y_1}) ) $  
	$\leq 2 \left( \dist(v^{(n)}_{x_0,y_1},v^{(n)}_{x_0,y_1+1})  + 
	\dist(v^{(n)}_{x_0,y_1+1}, \midp( v^{(n)}_{x_0+1,y_1+1},v^{(n)}_{x_0,y_1}) )  \right).$
	The first term is bounded by the argument right above. Noticing the boundedness of $d_{1,1}$ by the constant $C''$ we get that
	$\dist(v^{(n)}_{x_0,y_1+1}, \midp( v^{(n)}_{x_0+1,y_1+1},v^{(n)}_{x_0,y_1}))$
	$\leq  \dist(v^{(n)}_{x_0,y_1+1}, \midp( v^{(n)}_{x_0,y_1+1},v^{(n)}_{x_0+1,y_1})) + C''.$
	But, we also have $\dist(v^{(n)}_{x_0,y_1+1}, \midp( v^{(n)}_{x_0,y_1+1},v^{(n)}_{x_0+1,y_1}))$
	$= 1/2  \dist( v^{(n)}_{x_0,y_1+1},v^{(n)}_{x_0+1,y_1}),$ and the latter is bounded by using the triangle inequality.
	Iterative application then shows 
	$\dist(v^{(n)}_{x+1,y},v^{(n)}_{x,y}) \leq C''' $ and 
	also $\dist(v^{(n}_{x,y+1},v^{(n)}_{x,y}) \leq C'''$	
	which by the triangle inequality implies the boundedness of $\diam(v^{(n)}).$
	Applying this to $u^{(n_k)}$ implies the boundedness of $\diam(u^{(n_k)})$ 
	which yields a contradiction as desired.
	Hence, $\TV_2(u^{(n)})\to \infty$
	and we apply Theorem~\ref{thm:ExistenceCondR4SecondOrd} to conclude the assertion.
\end{proof}

\begin{proof}[Proof of Theorem~\ref{thm:ExistenceMinimizersTGV}]
	We apply Theorem~\ref{thm:ExistenceCondR4SecondOrd}. 
	In \cite{bredies2017total}, we have shown the lower semicontinuity of \eqref{eq:DefTGVreg}.
	Towards the other condition of Theorem~\ref{thm:ExistenceCondR4SecondOrd}, 	
	let $u^{(n)}$ be  a sequence such that  $\diam(u^{(n)}) \to \infty,$ 
	and such that we have 
	$\dist(u^{(n)}_{x_0,y_0},u^{(n)}_{x_0+1,y_0}) \leq C', $  
	$\dist(u^{(n)}_{x_1,y_1},u^{(n)}_{x_1,y_1+1}) \leq C', $ as well as 
	$\dist(u^{(n)}_{x_2,y_2},u^{(n)}_{x_2+1,y_2}) \leq C',$ $|y_2-y_0|=1,$ 
	for some $C'>0$ and all $n \in \mathbb N.$ 
	We have to show that  $\TGV_\lambda(u^{(n)})\to \infty.$		
	Towards a contradiction, assume that there is a subsequence $u^{(n_k)}$ of $u^{(n)}$ and $C''>0$ such that $\TGV_\lambda(u^{(n_k)}) \leq C''.$ 
	In the following, we show that, since $\TGV_\lambda(u^{(n_k)}) \leq C'',$ there is a 
	constant $C'''>0$ such that $\dist(u^{(n_k)}_{x+1,y},u^{(n_k)}_{x,y}) \leq C''' $ and 
	such that $\dist(u^{(n_k)}_{x,y+1},u^{(n_k)}_{x,y}) \leq C'''$ for all $x,y.$
	To this end, we first observe that the boundedness of the expressions in the first and second line of \eqref{eq:DefTGVreg} 
	implies that either the distance between any two neighboring items in the $x_0$th line and in the $y_0$th column remains bounded as follows. 
	The boundedness of $\dist(u^{(n_k)}_{x_0,y_0},u^{(n_k)}_{x_0+1,y_0})$ implies using
	$\dist(u^{(n_k)}_{x_0,y_0},v^{(n_k),1}_{x_0,y_0}) \leq \dist(u^{(n_k)}_{x_0,y_0},u^{(n_k)}_{x_0+1,y_0})  + C'',$ 
	($C''$ due to the first line of \eqref{eq:DefTGVreg})
	the boundedness of $\dist(u^{(n_k)}_{x_0,y_0},v^{(n_k),1}_{x_0,y_0}).$
	Using $\TGV_\lambda(u^{(n_k)}) \leq C'',$ the second line of \eqref{eq:DefTGVreg} yields the boundedness of $\dist(u^{(n_k)}_{x_0+1,y_0},v^{(n_k),1}_{x_0+1,y_0}).$  
	Then the first line of \eqref{eq:DefTGVreg} yields 
	$\dist(u^{(n_k)}_{x_0+1,y_0},u^{(n_k)}_{x_0+2,y_0})$
	$\leq \dist(u^{(n_k)}_{x_0+1,y_0},v^{(n_k),1}_{x_0+1,y_0}) + C''$
	which implies the boundedness of $\dist(u^{(n_k)}_{x_0+1,y_0},u^{(n_k)}_{x_0+2,y_0}).$
	Iterative application yields the boundedness of $\dist(u^{(n_k)}_{x,y_0},u^{(n_k)}_{x,y_0})$
	for any $x.$ The analogous statement follows for $x_1,y_1$
	and we get the validity for the ``cross'' of indices $(x,y)$ with either $x=x_0$ or $y=y_0$  
	with the respective directions along the ``cross''. 
	Now we invoke the additional assumption \eqref{eq:4fromCrossToAll}.
	By the argumentation above, we get the boundedness of $\dist(u^{(n_k)}_{x,y_2},u^{(n_k)}_{x,y_2})$ 
	for any $x.$ Then we may apply the triangle inequality to conclude the boundedness
	of $\dist(u^{(n_k)}_{x_0+1,y_0},u^{(n_k)}_{x_0+1,y_2})$ by that of the sum of 
	$\dist(u^{(n_k)}_{x_0+1,y_0},u^{(n_k)}_{x_0,y_0}),$  and 
	$\dist(u^{(n_k)}_{x_0,y_0},u^{(n_k)}_{x_0,y_2}),$ and 
	$\dist(u^{(n_k)}_{x_0+1,y_2},u^{(n_k)}_{x_0,y_2}).$
	Iterated application then yields the boundedness of
	$\dist(u^{(n_k)}_{x,y_0},u^{(n_k)}_{x,y_2})$ for all $x.$
	Then, for each fixed $x$ the boundedness of the $\TGV$ term, in particular the second line of \eqref{eq:DefTGVreg},
	(with the same argument as above for $x_0$ above) yields the boundedness of 
	$\dist(u^{(n_k)}_{x,y},u^{(n_k)}_{x,y+1})$ for all $x,y.$
	Again, applying the triangle inequality, also 
	$\dist(u^{(n_k)}_{x,y},u^{(n_k)}_{x+1,y})$ remains bounded for all $x,y.$
	This implies that $\diam(u^{(n_k)})$ is bounded which contradicts our assumption.
	Hence, $\TGV_\lambda(u^{(n)})\to \infty$
	and we may apply Theorem~\ref{thm:ExistenceCondR4SecondOrd} to conclude the assertion.
\end{proof}

{\small
	\bibliographystyle{plainnat}
	\bibliography{ManiCon}
}

\end{document}